\documentclass{siamart1116}

\usepackage{enumitem}
\setlist[description]{leftmargin=3.5mm}
\usepackage{amsfonts}
\usepackage{graphicx}
\usepackage{epstopdf}
\usepackage{algorithmic}
\ifpdf
  \DeclareGraphicsExtensions{.eps,.pdf,.png,.jpg}
\else
  \DeclareGraphicsExtensions{.eps}
\fi

\graphicspath{{img/},{img/plotObjective/}}

\numberwithin{theorem}{section}

\usepackage{tikz,pgfplots}
\usetikzlibrary{shapes}
\usepackage{graphicx}
\usepackage{algorithm,algorithmic}
\newcommand{\rottext}[1]{\rotatebox{90}{\hbox to 20mm{\hss #1\hss}}}
\newcommand\norm[1]{\left\lVert#1\right\rVert}
\usepackage{datetime}

\newcommand{\R}{\mathbb{R}}

\renewcommand{\norm}[2][]{\left\Vert#2\right\Vert_{#1}}

\newcommand{\kron}{\otimes}
\DeclareMathOperator*{\argmin}{arg\,min}                   
\renewcommand{\t} {^{\top}}                                

\newcommand{\diag}{\mathrm{diag}}

\renewcommand{\phi}{\mathbf{\varphi}}

\newcommand{\bfzero}{{\bf0}}

\newcommand{\bfd}{\mathbf{d}}

\newcommand{\bfT}{\mathbf{T}}
\newcommand{\bff}{\mathbf{f}}
\newcommand{\bfA}{\mathbf{A}}
\newcommand{\bfC}{\mathbf{C}}
\newcommand{\bfE}{\mathbf{E}}
\newcommand{\bfU}{\mathbf{U}}
\newcommand{\bfP}{\mathbf{P}}
\newcommand{\bfR}{\mathbf{R}}
\newcommand{\bfI}{\mathbf{I}}

\newcommand{\bfb}{\mathbf{b}}
\newcommand{\bfc}{\mathbf{c}}

\newcommand{\bfX}{\mathbf{X}}

\newcommand{\bfe}{\mathbf{e}}

\newcommand{\bfM}{\mathbf{M}}

\newcommand{\bfL}{\mathbf{L}}
\newcommand{\bfw}{\mathbf{w}}

\newcommand{\bfp}{\mathbf{p}}

\newcommand{\bfh}{\mathbf{h}}
\newcommand{\bfr}{\mathbf{r}}
\newcommand{\bfs}{\mathbf{s}}
\newcommand{\bfv}{\mathbf{v}}
\newcommand{\bfV}{\mathbf{V}}

\newcommand{\bfQ}{\mathbf{Q}}
\newcommand{\bfJ}{\mathbf{J}}
\newcommand{\bfS}{\mathbf{S}}
\newcommand{\bfSigma}{\mathbf{\Sigma}}

\newcommand{\x}{x}														
\renewcommand{\xi}[1]{\x_{#1}}								

\newcommand{\bfK}{\mathbf{K}}

\newcommand{\trace}{{\mathop{\mathrm{tr}}}}

\newcommand{\calI}{\mathcal{I}}
\newcommand{\calJ}{\mathcal{J}}

\newcommand{\calR}{\mathcal{R}}

\newcommand{\bfGamma}{{\boldsymbol{\Gamma}}}

\newcommand{\bfLambda}{{\boldsymbol{\Lambda}}}

\newcommand{\bfPi}{{\boldsymbol{\Pi}}}

\newcommand{\bfPsi}{{\boldsymbol{\Psi}}}

\newcommand{\bfdelta}{{\boldsymbol{\delta}}}

\newcommand{\bfvarepsilon}{{\boldsymbol{\varepsilon}}}

\newcommand{\bftheta}{{\boldsymbol{\theta}}}

\newcommand{\bflambda}{{\boldsymbol{\lambda}}}
\newcommand{\bfmu}{{\boldsymbol{\mu}}}

\newcommand{\bbE}{\mathbb{E}}

\newcommand{\bbR}{\mathbb{R}}

\newcommand{\e}   {\textnormal{e}}                         
\renewcommand{\i}   {\textnormal{i}}                         
\usetikzlibrary{arrows}

\newlength\iwidth
\newlength\iheight

\usepackage{tabularx}
\usepackage{graphicx}
\usepackage{array}
\usepackage{longtable}
\usepackage{verbatim}
\newcommand{\true}{{\rm true}}
\newdimen\iwidth
\newdimen\iheight

\theoremstyle{plain} 

\newcommand{\thf} {\tfrac{1}{2}}                            

\newcommand{\TheTitle}{Optimal Experimental Design for Constrained Inverse Problems}
\newcommand{\TheAuthors}{L. Ruthotto, J. Chung, and M. Chung}

\headers{OED for Constrained Inverse Problems}{\TheAuthors}

\title{{\TheTitle}\thanks{Submitted to the editors \today.
\funding{This work was partially supported by NSF DMS 1522599 (L. Ruthotto) and NSF DMS 1723005 (J. Chung and M. Chung).}}}

\author{
Lars Ruthotto\thanks{Department of Mathematics and Computer Science, Emory University, Atlanta, GA
  (\email{lruthotto@emory.edu}, \url{http://www.mathcs.emory.edu/\~lruthot/}).}
  \and
  Julianne Chung\thanks{Department of Mathematics, Computational Modeling and Data Analytics Division, Academy of Integrated Science, Virginia Tech, Blacksburg, VA
    (\email{jmchung@vt.edu}, \url{http://www.math.vt.edu/people/jmchung/}).}
  \and
  Matthias Chung\thanks{Department of Mathematics, Computational Modeling and Data Analytics Division, Academy of Integrated Science, Virginia Tech, Blacksburg, VA
    (\email{mcchung@vt.edu}, \url{http://www.math.vt.edu/people/mcchung/}).}
}

\usepackage{amsopn}

\begin{document}

\maketitle

\begin{abstract}
In this paper, we address the challenging problem of optimal experimental design (OED) of constrained inverse problems.
We consider two OED formulations that allow to reduce the experimental costs by minimizing the number of measurements. The first formulation assumes a fine discretization of the design parameter space and uses sparsity promoting regularization to obtain an efficient design.  The second formulation parameterizes the design and seeks optimal placement for these measurements by solving a small-dimensional optimization problem. We consider both problems in a Bayes risk as well as an empirical Bayes risk minimization framework.
For the unconstrained inverse state problem, we exploit the closed form solution for the inner problem to efficiently compute derivatives for the outer OED problem. 
The empirical formulation does not require an explicit solution of the inverse problem and therefore allows to integrate constraints efficiently.  A key contribution is an  efficient optimization method for solving the resulting, typically high-dimensional, bilevel optimization problem using derivative-based methods. To overcome the lack of non-differentiability in active set methods for inequality constraints problems, we use a relaxed interior point method.  To address the growing computational complexity of empirical Bayes OED, we parallelize the computation over the training models.  Numerical examples and illustrations from tomographic reconstruction, for various data sets and under different constraints, demonstrate the impact of constraints on the optimal design and highlight the importance of OED for constrained problems.
\end{abstract}



\section{Introduction} \label{sec:introduction}
The ability to optimally configure or design an experimental setup, or optimal experimental design (OED), can have significant benefits and improvements in a wide range of scientific and engineering applications \cite{Atkinson2007,Pukelsheim2006}.  Only recently has the focus shifted from OED for well-posed problems to OED for ill-posed inverse problems, where an unresolved challenge is how to efficiently include constraints on the state parameters (i.e., the solutions of the inverse problem) \cite{HaberEtAl2008OED,HaberEtAl2010}.  In this work, we investigate the impact of state constraints on the optimal design, and we propose a unified OED framework for linear inverse problems with linear equality and inequality constraints in the context of tomographic reconstruction.

Before introducing the OED problem, we first describe the discrete inverse problem.
Given a design parameter $\bfp \in \bbR^\ell$ that describes the experiment setup, the observations are given as
\begin{equation}
    \bfd(\bfp) = \bfM(\bfp) \bff_\true + \bfvarepsilon(\bfp),
    \label{eq:forwardmodel}
\end{equation}
where $\bff_\true \in \bbR^n$ is the exact model that we wish to reconstruct, $\bfM: \bbR^{\ell} \to \bbR^{m \times n}$ is the design-dependent forward operator with matrix $\bfM(\bfp)$ describing the parameter-to-observation map,
and $\bfvarepsilon(\bfp)\in\bbR^{m}$ represents additive noise.
Some examples of $\bfM(\bfp)$ include the map from image to sinogram in tomography (see, e.g., \cite{hsieh2003computed,kak2001principles,natterer2001mathematics}) or the map onto an observation space of the solution of a partial or ordinary differential equation~(see, e.g., \cite{Arridge1999,CheneyEtAl1999,Haber2015}). In this work we assume that the measurements $\bfd(\bfp) \in\bbR^m$ depend on the underlying design of the experiment
(e.g., $\bfp$ may determine the positions of the sources and/or detectors or represent the times at which measurements are taken).  The noise can come from various sources, e.g., measurement errors, modeling errors, and numerical rounding errors. For simplicity we assume that $\bfvarepsilon(\bfp)$ is normally distributed with zero mean and known symmetric positive definite covariance matrix $\bfGamma_\bfvarepsilon(\bfp)\in \bbR^{m \times m}$ for any design $\bfp$.

We focus on ill-posed inverse problems where regularization in the form of prior knowledge is required to compute stable, reasonable approximations of $\bff_\true$ \cite{hansen2010discrete}.  Following a Bayesian framework~\cite{Calvetti:2007dq,Kaipio:2006gx}, we treat $\bff_{\true}$ as a random variable with a truncated multivariate normal distribution with probability density given as
\begin{equation*}
  \pi(\bff_\true) =
    \begin{cases}
      c\,\e^{-\tfrac{\gamma^2}{2} (\bff_{\true}-\bfmu)\t\bfGamma_\bff^{-1}(\bff_{\true}-\bfmu) } , & \text{if } \bfC_{\rm e} \bff_{\true} - \bfc_{\rm e} = \bfzero$ and $\bfC_{\rm i} \bff_{\true} - \bfc_{\rm i} \geq \bfzero, \\
      0, & \text{otherwise,}
    \end{cases}
\end{equation*}
with appropriate constants $c,\gamma > 0$ and given positive definite covariance matrix $\bfGamma_\bff \in \bbR^{n\times n}$ and mean $\bfmu \in \bbR^n$. Here, the pairs $\bfC_\e \in \R^{m_e \times n}, \bfc_\e\in \R^{m_e}$ and $\bfC_{\rm i}\in \R^{m_i \times n}$, $\bfc_{\rm i}\in \R^{m_i}$ define linear equality and inequality constraints on $\bff$, respectively.
Denoting by $\bfI_n \in \R^{n\times n}$ the identity matrix, for example, bound constraints correspond to choosing $\bfC_{\rm i} = [\bfI_n; -\bfI_n] \in \R^{2n\times n}$ and $\bfc_{\rm i} = [\bff_{\rm L}; -\bff_{\rm H}]\in\R^{2n}$ where $\bff_{\rm L}$ and $\bff_{\rm H}$ contain lower and upper bounds respectively (e.g., non-negativity constraints are reasonable when reconstructing density images).  Furthermore, setting $\bfC_\e = \bfe^\top$, where $\bfe$ is a vector of all ones, allows one to fix the integral (or mass) of
$\bff$,
which might be helpful in applications such as emission tomography.  A wide range of prior knowledge can be included to estimate $\bff_{\true}$ in this formulation.  The unconstrained case reduces to a simple Gaussian distribution on the entire domain \cite{calvetti2007introduction} and can provide insight in the Bayes formulation.

For fixed $\bfp,$ the maximum a-posteriori (MAP) estimate $\widehat \bff(\bfp)$ provides an estimator to~\eqref{eq:forwardmodel} and can be obtained by minimizing the negative log likelihood of the posterior probability distribution function, i.e.,
\begin{gather}
  \begin{split}\label{eq:ip}
    \widehat \bff(\bfp) = \argmin_{\bff} \tfrac{1}{2}\norm[\bfGamma^{-1}_\bfvarepsilon(\bfp)]{\bfM(\bfp)\bff -\bfd(\bfp)}^2 + \tfrac{\gamma^2}{2}\norm[2]{\bfL(\bff -\bfmu)}^2\\
    \text{subject to } \ \bfC_{\rm e} \bff - \bfc_{\rm e} = \bfzero \text{ and } \bfC_{\rm i} \bff - \bfc_{\rm i} \geq \bfzero\,,
  \end{split}
\end{gather}
where $ \bfL\t\bfL = \bfGamma_\bff^{-1}$.  In general, obtaining $\widehat \bff(\bfp)$ requires solving a constrained optimization problem, which in this case is a convex quadratic programming problem. In the absence of inequality constraints, the optimal solution to~\eqref{eq:ip} depends linearly on the data, which can be used for efficient optimization of the design parameter; see~\cite{HaberEtAl2008OED}. In the presence of inequality constraints however, the optimality condition becomes nonlinear in the data and, thus, solving~\eqref{eq:ip} becomes computationally challenging. We assume a unique minimizer of the constrained optimization problem exists.

In this work we consider optimal design problems that aim at selecting an experimental design that not only leads to most accurate estimates of $\bff_\true$ but also minimizes experimental costs.  Finding a balance between these often conflicting goals can be challenging.  For example, in tomographic reconstruction, obtaining more projection data may lead to higher resolution image reconstructions, but more scans may lead to more harmful radiation to the patient, longer scan times, and higher operational costs.

For our OED problem, the goal is to obtain design parameters $\widehat\bfp$ that minimize the Bayes risk, i.e., the expected value of the mean squared reconstruction errors, while simultaneously minimizing measurement costs.  This requires solving an optimization problem of the form,
\begin{equation}
    \label{eq:designproblem}
    \begin{split}
    \min_{\bfp \in \Omega}& \quad \calJ(\bfp) =  \bbE \ \tfrac{1}{2}\norm[2]{\widehat\bff(\bfp) - \bff_\true}^2 + \calR(\bfp),
    \end{split}
\end{equation}
where $\Omega\subset\R^p$ is the set of feasible design parameters, $\bbE$ is the expected value where $\bff_\true$ and $\bfvarepsilon(\bfp)$ are random variables, and the functional $\calR$ encodes the measurement costs.  Two formulations in the context of tomographic reconstruction will be described in Section~\ref{sec:problem_setup}.

The literature on OED methodologies is vast, with a range of techniques tailored to various
optimality criteria in both Bayesian \cite{chaloner1995bayesian} and non-Bayesian settings.  Classic OED works include \cite{Atkinson2007,pazman1986foundations,Pukelsheim2006}, but a recent large effort has been made to develop efficient algorithms for obtaining designs that minimize a loss function of the Fisher information matrix and to new applications, e.g., in biology and exploratory drilling \cite{HaberEtAl2008OED,HaberEtAl2010,haber2012numerical,Nowak2010}.
Many of these approaches follow an empirical Bayes risk minimization framework and exploit the case where the model parameters, here $\widehat \bff(\bfp)$, depend linearly on the observables $\bfd(\bfp).$  This is not necessarily the case for the constrained inverse problems of interest here.  Extensions to nonlinear problems have been considered in \cite{alexanderian2016fast,Chung2012,huan2013simulation}, and an approach based on consistent Bayesian inference was developed in~\cite{butler2017consistent} and used for OED in \cite{walsh2017optimal}.

The primary goal in many design problems is to enforce sparse sampling in the design parameters.  Sparsity enforcing regularization in the context of ODE was considered in \cite{alexanderian2014optimal,Chung2012,haber2012numerical,haber2015optimal}, but no additional state constraints were considered.  The main contribution of this work is the inclusion of state constraints in OED frameworks, which is a critical yet missing piece in the shift from OED for well-posed to OED for ill-posed problems.  Two open questions include (1) how does the inclusion of state constraints \emph{impact} the optimal design? And (2) how does one \emph{efficiently} incorporate constraints in an OED framework?
In this paper, we address both of these questions by developing and investigating a unified OED framework for constrained inverse problems.

An overview of this paper is as follows.  In Section~\ref{sec:problem_setup} we describe two problem formulations for OED in the context of tomographic reconstruction, one of which assumes a fine discretization of the parameter space and enforces sparsity of the design and the other assumes a fixed number of design parameters and seeks optimal locations for these measurements.  In Section~\ref{sec:problemformulation} we investigate various problem formulations for the OED problems that demonstrate the range of problems and constraints that can be addressed. We describe efficient computational approaches in Section~\ref{sec:computational}.
In Section~\ref{sec:numerical_experiments}, we demonstrate the impact of state constraints on the design, which is something that has not been investigated before, and we provide numerical results that demonstrate the effectiveness of our approach.  Conclusions are provided in Section~\ref{sec:conclusions}.

\section{Motivating Application: Tomographic Reconstruction}
\label{sec:problem_setup}
Although the methods described in this paper extend to more general problems and applications, we focus our discussion on two problem setups in the context of tomographic reconstruction.

Tomography is a widely used imaging technique where penetrating waves (e.g., x-ray or acoustic waves) transverse an object and are collected or detected~\cite{hsieh2003computed,kak2001principles,natterer2001mathematics}.
Oftentimes, multiple projections are made as the wave source rotates around the object.  Then given these observed measurements, the goal of the inverse problem is to reconstruct the interior properties (e.g., densities) of the object at different locations; see Figure~\ref{fig:tomo}.
\begin{figure}[bthp]
 \begin{center}
  \input{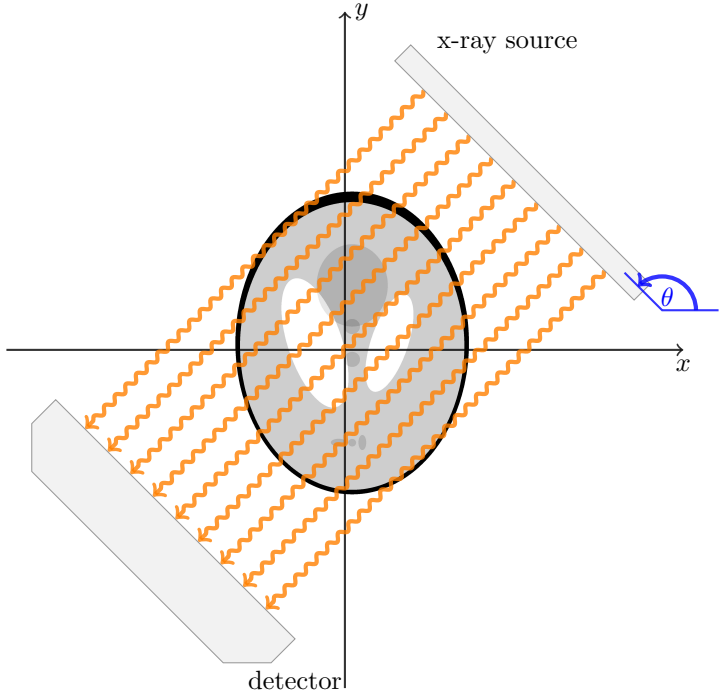}
  \end{center}
  \caption{Illustration of a tomography experiment.}\label{fig:tomo}
\end{figure}
Mathematically, we can model the transmission process for one projection via a (sparse) matrix $\bfT \in \bbR^{n_r \times n}$ where $n_r$ is the number of rays.  Then the noise-free projection data obtained by rotating the source $\theta$ degrees clockwise can be modeled as $\bfd(\theta) = \bfT \bfR(\theta) \bff_\true\,,$
where $\bfR(\theta)\in \bbR^{n \times n}$ rotates the object by $\theta$ degrees counterclockwise.  Typically we assume $\theta \in [0, 180].$ Next we describe two OED problems.

\paragraph{OED Problem A}
In the first scenario, we assume that a fine discretization of the set of projection angles $\bftheta = [\theta_1, \ldots, \theta_\ell]\t$ is given and aim at identifying the angles that provide the most important measurements. To this end, we introduce the design parameters $\bfp\in \bbR^\ell$ with $\bfp \geq \bfzero$ whose components encode the importance of each measurement.
Define the ordered index set $\calI(\bfp) = \{i \, : \, p_i > 0 \}$ and denote $k(\bfp)$ to be the cardinality of $\calI(\bfp).$ We define matrix $\bfE(\bfp) \in \bbR^{\ell \times k(\bfp)}$ to contain the $i$-th standard basis vectors for $i \in \calI(\bfp).$ Then, $\bfM(\bfp)$ and $\bfd(\bfp)$ are given by
\begin{equation}
\bfM(\bfp) =
\bfP\left( \bfI_\ell \kron \bfT \right) \begin{bmatrix} \bfR(\theta_1)\\ \vdots \\ \bfR(\theta_\ell) \end{bmatrix}\,, \qquad \text{and}  \qquad \bfd(\bfp) = \bfP\bfd,
\end{equation}
where $\bfP = \left(\bfE(\bfp)\t \diag(\bfp)\right) \otimes \bfI_{n_r} \in \bbR^{k(\bfp) n_r \times \ell n_r}.$  Notice that $\bfE(\bfp)\t \diag(\bfp)$ essentially extracts all of the non-zero rows of $\diag(\bfp).$ We assume that $\bfGamma_\bfvarepsilon(\bfp)^{-1}  = \sigma^2 \bfI_{k(\bfp) n_r}$.

Since we would like to keep the number of projection angles low, we incorporate a sparsity-inducing prior on the design parameters.  That is, we use the $\ell_1$-norm $\calR(\bfp) = \beta \norm[1]{\bfp}$ with $\beta>0$ large enough so that we can get a relaxation of $\|\cdot\|_0$ and enforce sparsity in the design parameters \cite{candes2006stable,candes2005decoding}.
  In the Bayesian framework, this is equivalent to imposing a Laplace prior distribution on $\bfp$ (see e.g., \cite{huang2013optimal,johnson1972distributions}).  Furthermore, we assume $\Omega = \R^{\ell}_+$ (non-negative orthant) such that continuous optimization methods can be used.

In summary, \emph{OED problem A}
can be written as the bilevel optimization problem
\begin{gather}
		\min_{\bfp \geq \bfzero} \quad \bbE \ \tfrac{1}{2}\norm[2]{\widehat\bff(\bfp) - \bff_\true}^2 + \beta \norm[1]{\bfp} \label{eq:designproblemA}\\
		\text{subject to} \nonumber\\
		\hat \bff(\bfp) = \argmin_{\bff} \tfrac{1}{2}\norm[2]{\bfM(\bfp)\bff -\bfd(\bfp)}^2 + \tfrac{\alpha^2}{2}\norm[2]{\bfL(\bff -\bfmu)}^2 \label{eq:ipA}\\
		\text{subject to } \ \bfC_{\rm e} \bff - \bfc_{\rm e} = \bfzero \text{ and } \bfC_{\rm i} \bff - \bfc_{\rm i} \geq \bfzero\,, \nonumber
\end{gather}
where $\alpha = \gamma/\sigma$ and zero values in $\bfp$ correspond to zeroing out measurements for the corresponding angle.
After computing such an $\ell_1$-regularized design, the solution can be used to identify important components, and a second optimization can be done to optimize the weights of the non-zero components (see~\cite{HaberEtAl2008OED}). While  the sparsity of the design generally depends on the choice of $\beta$ (e.g., the larger $\beta$ the fewer non-zero elements in the design vector) a relevant issue in some applications is identifying $\beta$ so that a design with a given number of measurements is obtained.

\paragraph{OED Problem B} In the second scenario, we present a method that optimizes the design parameters for a fixed number of measurements. Suppose $\bfp\in \Omega \subset \bbR^\ell$, $\ell < m,$ e.g., for tomography it contains angles $ \bfp_{\rm L}\leq \bfp \leq \bfp_{\rm H}$ corresponding to locations of the sources.  Then in the inverse problem,
\begin{equation}
\bfM(\bfp) = \left( \bfI_\ell \kron\bfT \right) \bfR(\bfp), \quad \mbox{ where } \quad \bfR(\bfp) = \begin{bmatrix} \bfR(p_1) \\ \vdots \\ \bfR(p_\ell) \end{bmatrix}
\end{equation}
 and $\bfGamma_\bfvarepsilon = \sigma^2\bfI_m$. Assuming that no additional regularization is required for $\bfp$ (other than the above described reparameterization), \emph{OED problem B} reads
\begin{gather}
		\min_{\bfp \in \Omega} \quad \bbE \ \tfrac{1}{2}\norm[2]{\widehat\bff(\bfp) - \bff_\true}^2 \label{eq:designproblemB}\\
		\text{subject to} \nonumber\\
		\hat \bff(\bfp) = \argmin_{\bff} \tfrac{1}{2}\norm[2]{\bfM(\bfp)\bff -\bfd(\bfp)}^2 + \tfrac{\alpha^2}{2}\norm[2]{\bfL(\bff -\bfmu)}^2 \label{eq:ipB}\\
		\text{subject to } \ \bfC_{\rm e} \bff - \bfc_{\rm e} = \bfzero \text{ and } \bfC_{\rm i} \bff - \bfc_{\rm i} \geq \bfzero \,. \nonumber
	\end{gather}

Note that mathematically OED problem A and B differ in the particular choice of $\bfM(\bfp)$, $\bfd(\bfp)$ and the sparsity regularization term.  We refer to problems~\eqref{eq:designproblemA} and~\eqref{eq:designproblemB} as the design problems and problems~\eqref{eq:ipA} and~\eqref{eq:ipB} as the underlying inverse or state problems.

\section{Design Problem Formulations}
\label{sec:problemformulation}
In this section, we consider OED problems~\eqref{eq:designproblemA} and~\eqref{eq:designproblemB} and discuss two problem formulations, each of which may have advantageous properties depending on problem assumptions and constraints.

\paragraph{Bayes risk minimization} Notice that the expected values in~\eqref{eq:designproblemA} and~\eqref{eq:designproblemB} are defined in terms of the distributions of $\bff_\true$ and $\bfvarepsilon(\bfp)$.  For problems where such knowledge is available or can be well approximated (e.g., by the sample mean or sample covariance), we investigate Bayes risk minimization for both design problems.  This approach assumes that no inequality constraints are included on the inverse state problems~\eqref{eq:ipA} and~\eqref{eq:ipB}. While the theory can be extended to equality constrained problems, we consider unconstrained inverse problems for simplicity. Under these assumptions the MAP estimate is given by
\begin{equation} \label{eq:lsproblem}
	\widehat\bff(\bfp) = \bfQ(\bfp)^{-1}\left( \bfM(\bfp)\t \left(\bfM(\bfp)\bff_{\rm true} + \bfvarepsilon(\bfp)\right)+\alpha^2 \bfL\t \bfL \bfmu \right),
\end{equation}
where $\bfQ(\bfp) = \bfM(\bfp)\t \bfM(\bfp) + \alpha^2 \bfL\t\bfL\,.$
Then the design objective (for convenience omitting the design costs expressed by $\calR$) can be written as
\begin{align*}
\calJ(\bfp) & = \tfrac{1}{2} \bbE \norm[2]{\bfQ(\bfp)^{-1} \left( \bfM(\bfp)\t \left(\bfM(\bfp)\bff_{\rm true} + \bfvarepsilon(\bfp)\right)+\alpha^2 \bfL\t \bfL \bfmu \right) - \bff_{\rm true}}^2\\
 & = \tfrac{1}{2} \bbE \norm[2]{\left(\bfQ(\bfp)^{-1} \bfM(\bfp)\t \bfM(\bfp) -\bfI_n \right) \bff_{\rm true} + \bfQ(\bfp)^{-1} (\bfM(\bfp)\t \bfvarepsilon(\bfp)+\alpha^2 \bfL\t \bfL \bfmu)}^2\,.
\end{align*}
Let $\bfK(\bfp) = \bfQ(\bfp)^{-1} \bfM(\bfp)\t \bfM(\bfp) -\bfI_n$ and denote $\trace(\cdot)$ to be the trace of a matrix, then by utilizing the quadratic form property $\bbE(\bfdelta\t\bfLambda\bfdelta) = \bfmu_\bfdelta\t\bfLambda \bfmu_\bfdelta + \trace(\bfLambda\bfGamma_\bfdelta)$
along with the above assumptions, we get
\begin{align*}
2 \calJ(\bfp) =\ & \bbE\Big( \bff_{\rm true}\t \bfK(\bfp)\t\bfK(\bfp) \bff_\true +2 \bff_{\rm true}\t \bfK(\bfp)\t \bfQ(\bfp)^{-1} \left(\bfM(\bfp)\t \bfvarepsilon(\bfp)+\alpha^2 \bfL\t \bfL \bfmu \right) \\
+& \ \left(\bfM(\bfp)\t \bfvarepsilon(\bfp)+\alpha^2 \bfL\t \bfL \bfmu\right)\t \bfQ(\bfp)^{-\top} \bfQ(\bfp)^{-1} \left(\bfM(\bfp)\t \bfvarepsilon(\bfp)+\alpha^2 \bfL\t \bfL \bfmu\right) \Big)\\
=\ & \bfmu\t \bfK(\bfp)\t\bfK(\bfp) \bfmu + \gamma^{-2} \trace\left(\bfK(\bfp)\t\bfK(\bfp) \bfGamma_\bff\right) + 2 \alpha^2 \bfmu\t \bfK(\bfp)\t \bfQ(\bfp)^{-1} \bfL\t \bfL \bfmu \\
+& \ \sigma^{-2} \trace\left(\bfM(\bfp) \bfQ(\bfp)^{-\top} \bfQ(\bfp)^{-1} \bfM(\bfp)\t\right) + \alpha^4 \bfmu\t \bfL\t \bfL \bfQ(\bfp)^{-\top} \bfQ(\bfp)^{-1} \bfL\t \bfL \bfmu.
\end{align*}
Notice that the first, third, and fifth term sum up to $\norm[2]{\left(\bfK(\bfp)+\alpha^2 \bfQ(\bfp)^{-1} \bfL\t \bfL\right)\bfmu}^2 = 0$.  Thus, we get
\begin{align*}
\calJ(\bfp) & = \tfrac{1}{2\gamma^2 } \norm[\rm F]{\bfK(\bfp)\bfL^{-1}}^2 +  \tfrac{1}{2\sigma^2} \norm[\rm F]{\bfQ(\bfp)^{-1} \bfM(\bfp)\t}^2\\
& =   \tfrac{1}{2\sigma^2} \norm[\rm F]{\bfQ(\bfp)^{-1} \begin{bmatrix} \bfM(\bfp) \\ \alpha \bfL \end{bmatrix}\t  }^2 \\
& =   \tfrac{1}{2\sigma^2} \norm[\rm F]{\bfM_\alpha^\dagger (\bfp) }^2\,,
\end{align*}
where $\bfM_\alpha(\bfp) = \begin{bmatrix} \bfM(\bfp) \\ \alpha \bfL \end{bmatrix}$, the $\dagger$ in the last equation denotes the Moore-Penrose pseudoinverse, and $\norm[\rm F]{\,\cdot\,}$ denotes the Frobenius norm.  Recall that the goal of OED is to find design parameters that minimize $\calJ(\bfp)$.
Thus, the significance of this result is that in the Bayes formulation with an unconstrained state problem, the optimal design parameters will correspond to a regularized coefficient matrix that has smallest pseudoinverse in the Frobenius norm sense (with an additional regularization term for OED problem A).  Due to the nonlinear dependence on $\bfp,$ the design problem does not admit a closed-form solution; however, numerical methods and computational simplifications can be used to obtain optimal parameters (see Section~\ref{sec:computational}).

\paragraph{Empirical Bayes risk minimization} For problems where distributions of $\bff_\true$ and $\bfvarepsilon(\bfp)$ may be unknown or not obtainable, we consider empirical Bayes risk design problems, where training or calibration data $\bff_\true^{(k)}$, $k = 1, \ldots, K$ are used to approximate the expected value. Such data are often readily available, and in the case of tomography provide a clear understanding of how images may look.  Stochastic programming methods such as stochastic average approximation (SAA) or stochastic approximation (SA) can be used to incorporate these training data in OED problems A and B. Solving stochastic optimization problems often requires computationally intensive techniques in order to obtain a good approximation of the expected value \cite{carlin1997bayes,chung2017stochastic,shapiro2009lectures}.  Here we consider an SAA approach, but remark that for very large numbers of training data, this approach may not be feasible.  However, a benefit of this formulation, compared to the Bayes risk minimization procedure described above, is the ability to incorporate constraints on the state problem and take advantage of existing constrained optimization algorithms.

We follow the \emph{empirical} Bayes approach in~\cite{Chung2012,HaberEtAl2008OED,HaberEtAl2010}  that treats given training data as samples from the distribution and uses the sample mean to approximate the expected value.
Note that, due to the presence of the regularization, the estimator given by~\eqref{eq:ip} will be biased unless the true solution is in the nullspace of $\bfL$. Thus, design choices cannot solely be based on properties of the forward operator $\bfM(\bfp)$ (see also the the discussion in~\cite{HaberEtAl2008OED}).

Assume that we are given a set of training data that consists of $N$ true models $\bff^{(1)}_\true, \ldots, \bff^{(N)}_\true \in\R^n$. For a fixed design parameter $\bfp$,  we can simulate datasets $\bfd^{(1)}(\bfp),\ldots,\bfd^{(N)}(\bfp) \in\R^m$ using~\eqref{eq:forwardmodel} and obtain reconstructions $\widehat{\bff}^{(1)}(\bfp),\ldots,\widehat{\bff}^{(N)}(\bfp) $ by solving the constrained inversion problem~\eqref{eq:ip} using that data.
Then we can approximate the Bayes risk OED problem~\eqref{eq:designproblem} with the following empirical Bayes risk OED problem,
\begin{equation}\label{eq:OED}
	\begin{split}
		\min_{\bfp \in \Omega} &\quad   \calJ_N(\bfp) =  \tfrac{1}{2N} \sum_{i=1}^N \left \| \widehat{\bff}^{(i)}(\bfp) - \bff_\true^{(i)} \right\|^2 +  \calR(\bfp)\\
		\text{ subject to} &  \quad \widehat{\bff}^{(i)}(\bfp) \text{ solves}~\eqref{eq:ip} \text{ for data } \bfd^{(i)}(\bfp).
	\end{split}
\end{equation}

The bilevel optimization problem~\eqref{eq:OED} could be solved as a large constrained optimization problem.  Instead we follow a technique commonly used in the PDE constrained optimization literature where we eliminate the constraint by solving for $\widehat{\bff}^{(i)}(\bfp), $ for $i=1,2,\ldots,N$ yielding the \emph{reduced} problem. Although both approaches have their merits, the reduced problem might be more attractive, especially if $N\gg1$.  Furthermore, the reduced problem allows for parallel computing, since the constraints can be eliminated independently.

Assuming that $\Omega \subset \R^{\ell}$ is closed and convex,~\eqref{eq:OED} can be solved, for example, using a Projected Steepest Descent or Projected Gauss-Newton method (see also~\cite{HaberEtAl2010}).  In both cases, we need to compute the gradient of the design objective function,
\begin{equation}
	\label{eq:OEDgradient}
	\nabla_{\bfp}  \calJ_N(\bfp) = \tfrac{1}{N} \sum_{i=1}^N \nabla_{\bfp}\widehat{\bff}^{(i)} \left(\widehat{\bff}^{(i)}(\bfp) - \bff_\true^{(i)} \right) + \nabla_{\bfp} \calR\,,
\end{equation}
where $\nabla_{\bfp}\widehat{\bff}^{(i)}$ contains the derivatives of the reconstructed image $\widehat{\bff}^{(i)}(\bfp)$ with respect to the design parameters.
For OED problem A, where $\calR(\bfp) =\beta \|\bfp\|_1$ and $\bfp \geq \bf0$, computing the gradient for the regularization term is straightforward.
However, the more challenging computation involves the sensitivity matrix $\nabla_{\bfp}\widehat{\bff}^{(i)}$ that is solver-dependent and outlined in the following section.

\section{Computational approaches for OED}
\label{sec:computational}

In this section, we describe efficient computational approaches for computing solutions to OED problems A and B.

\paragraph{Bayes risk minimization}
We begin with computational simplifications that can be used for the Bayes risk minimization problem.  First we consider the case where $\bfL =\bfI_n$ (standard Tikhonov regularization).  Assume $\bfM(\bfp)\in \bbR^{m \times n}$ with rank $r \leq \min(m,n)$ and let $\bfM(\bfp) = \bfU(\bfp) \bfSigma(\bfp) \bfV(\bfp)\t$ be the SVD where the non-zero diagonal elements of $\bfSigma(\bfp)$ are $\sigma_1(\bfp) \geq \ldots \geq \sigma_r(\bfp)$.
Then, with some algebraic manipulations, it can be shown that design objective (again omitting the regularization term for convenience) can be written as
\begin{align*}
 \calJ(\bfp) & =  \tfrac{1}{2\sigma^2} \norm[\rm F]{(\bfSigma(\bfp)\t \bfSigma(\bfp) + \alpha^2 \bfI_n)^{-1} \begin{bmatrix}\bfSigma(\bfp)\t & \alpha \bfI_n \end{bmatrix} }^2 \\
 & = \tfrac{1}{2\sigma^2} \left(\sum_{i=1}^r \frac{1}{\sigma_i(\bfp)^2 + \alpha^2} +  \frac{n-r}{\alpha^2} \right)\,.
\end{align*}
Note that the second term should not be ignored since the rank of $\bfM(\bfp)$ may depend on $\bfp.$

For the general case where $\bfL\neq \bfI_n,$ the generalized SVD could be used to simplify the objective function as
\begin{equation*}
 \calJ(\bfp)  =  \tfrac{1}{2\sigma^2} \norm[\rm F]{ \bfX(\bfp)^{-1}(\bfSigma(\bfp)\t \bfSigma(\bfp) + \alpha^2 \bfPsi(\bfp)\t \bfPsi(\bfp))^{-1/2} }^2\,,
 \end{equation*}
where
$$\bfM(\bfp) = \bfU(\bfp) \bfSigma(\bfp) \bfX(\bfp) \t \quad \mbox{and} \quad \bfL = \bfV(\bfp) \bfPsi(\bfp) \bfX(\bfp)\t$$
and $\bfU(\bfp)\in \bbR^{m\times m}$ and $\bfV(\bfp)\in \bbR^{n\times n}$ are orthogonal matrices, $\bfSigma(\bfp)\in \bbR^{m\times n}$ and $\bfPsi(\bfp)\in \bbR^{n\times n}$ are diagonal matrices and $\bfX(\bfp) \in \bbR^{n \times n}$ is a nonsingular matrix \cite{stewart2001matrix}.  A simpler approach for $\bfL\neq \bfI_n$ is to compute the singular values of $\bfM_\alpha(\bfp)$ denoted $\sigma_{\alpha,1}(\bfp), \ldots, \sigma_{\alpha,n}(\bfp)  > 0$, in which case
$$\calJ(\bfp) =  \tfrac{1}{2\sigma^2} \sum_{i=1}^n \left(\sigma_{\alpha,i}(\bfp)\right)^{-2}\,.$$

Next we exploit the specific structure of OED problem A to analyze the dependence on $\bfp$. Assume that all $\bfp >\bfzero$ so that $\bfE(\bfp) = \bfI_\ell$, then
$$\bfM(\bfp) = (\diag(\bfp) \otimes \bfI_{n_r}) \begin{bmatrix} \bfT \bfR(\theta_1) \\ \vdots \\ \bfT \bfR(\theta_\ell)\end{bmatrix} = \begin{bmatrix} p_1 \bfT \bfR(\theta_1) \\ \vdots \\ p_\ell \bfT \bfR(\theta_\ell)\end{bmatrix}\,. $$
The singular values of $\bfM(\bfp)$ are given in the following Lemma.
\begin{lemma}
 Let $\bfA_j \in \bbR^{m \times n}$ for $j = 1, \ldots, \ell$, then the unsorted singular values of $(\diag(\bfp) \kron \bfI) \begin{bmatrix} \bfA_1 \\ \vdots \\ \bfA_\ell\end{bmatrix}$ are given by $\sqrt{\sum_{j=1}^\ell \left(\sigma_i^{(j)}\right)^2 p_j^2}, i=1, \ldots, n$ where $\sigma_i^{(j)}$ is the $i$-th singular value of $\bfA_j$\,.
\end{lemma}
\begin{proof}
  The result follows from the fact that the eigenvalues of $\sum_{j=1}^\ell p_j^2 \bfA_j\t \bfA_j$ are given as $\sum_{j=1}^\ell \left(\sigma_i^{(j)}\right)^2 p_j^2$.
\end{proof}
Thus, if we let $\bfA_j = \bfT \bfR(\theta_j)$ and define $\bfPi \in \bbR^{n \times \ell}$ such that the $j$-th column of $\bfPi$ contains the squares of the singular values of $\bfA_j$, then the squares of the singular values of $\bfM(\bfp)$ are given in the vector $\bfh = \bfPi \diag(\bfp)\bfp.$ This formulation enables us to compute the gradient as
\begin{equation}
 \nabla_\bfp \calJ(\bfp) = - \sigma^{-2} \diag(\bfp) \bfPi\t \bfh_\alpha + \beta {\rm sign}(\bfp),
\end{equation}
where $\bfh_\alpha = \begin{bmatrix} (h_1 + \alpha^2)^{-2}, \ldots, (h_n + \alpha^2)^{-2}\end{bmatrix}\t$ with $h_i$ being the $i$-th element of $\bfh$.
Notice that a critical point occurs if $\bfp=\bfzero$ (ignoring the non-differentiability of $\norm[1]{\bfp}$ at $\bfp = \bfzero$), but this implies that $\bfM(\bfp)=\bfzero$ which is not reasonable.  Thus, to obtain optimal design parameters for OED problem A, a nonlinear solver must be used.

\paragraph{Empirical Bayes risk minimization}
In the absence of a closed form expression for the MAP estimate (e.g., due to the presence of inequality constraints) we follow an SAA approach and use gradient-based optimization methods to solve the empirical Bayes risk problem~\eqref{eq:OED}. To ensure differentiability, we propose using interior point methods for solving the constrained inverse state problems~\eqref{eq:ipA} and~\eqref{eq:ipB}.  Although active set methods could also be used for solving problems such as~\eqref{eq:ip}, we prefer interior point methods because the ultimate goal is fast optimization of the design, and interior point methods enable fast sensitivity computations, i.e., methods for computing the derivatives of the reconstructed solution with respect to the (possibly relaxed) design parameters. Before deriving the sensitivities we describe the interior point method used in our experiments, which is a standard primal dual interior point method for quadratic programming based on Mehrotra's predictor-corrector approach; a more detailed description can be found, e.g., in~\cite[Ch.16]{NocedalWright2006}.

We first rewrite the constrained optimization problem as a quadratic program
\begin{equation}
	\min_{\bff}  \thf \bff^\top \bfQ(\bfp) \bff + \bfb(\bfp)^\top \bff
	\quad\text{ subject to }\quad \bfC_\e \bff - \bfc_\e = \bfzero, \quad \bfC_\i \bff - \bfc_\i \geq \bfzero \,,
\end{equation}
where $\bfQ(\bfp) = \bfM(\bfp)^\top \bfM(\bfp) + \alpha^2 \bfL^\top \bfL$ and $\bfb(\bfp) = -\bfM(\bfp)^\top \bfd(\bfp)-\alpha^2 \bfL\t \bfL \bfmu$.
To deal with the linear inequality constraints, we introduce slack variables $\bfs \in \R^{m_i}$, yielding the equivalent problem
\begin{equation}\label{eq:QPslack}
	\min_{\bff,\bfs} \thf \bff^\top \bfQ(\bfp) \bff + \bfb(\bfp)^\top \bff
	\quad\text{ subject to }\quad \bfC_\e \bff - \bfc_\e = \bfzero, \quad \bfC_\i \bff - \bfc_\i - \bfs = \bfzero,\quad \bfs \geq \bfzero.
\end{equation}
This is a convex quadratic optimization problem, whose objective function is strictly convex if ${\rm null}(\bfM(\bfp)) \cap {\rm null}(\bfL) = \emptyset$.  The Lagrangian is given by
\begin{equation}\label{eq:Lag}
\mathcal{L}(\bff, \bflambda_\e,\bfs, \bflambda_\i) = \thf \bff^\top \bfQ(\bfp) \bff + \bfb(\bfp)^\top \bff - \bflambda_\e^\top (\bfC_\e \bff - \bfc_\e) - \bflambda_\i^\top (\bfC_\i \bff - \bfc_\i - \bfs).
\end{equation}
Necessary and sufficient conditions for a global minimizer are the KKT conditions. Here, we consider the perturbed conditions for some centrality parameter $\delta \geq 0$,
\begin{equation}\label{eq:KKTCond}
	F(\bff,\bflambda_\e,\bfs,\bflambda_\i,\delta,\mu) = \left[
		\begin{array}{c}
			\bfQ(\bfp) \bff + \bfb(\bfp) - \bfC_\e^\top \bflambda_\e - \bfC_\i^\top \bflambda_\i \\
			\bfC_\e \bff - \bfc_\e \\
			\bfC_\i \bff - \bfc_\i - \bfs\\
			\bfS \bfLambda_\i \bfe - \delta \mu \bfe
		\end{array}
	\right] = \bfzero,\quad  \quad  \bflambda_\i, \bfs \geq \bfzero.
\end{equation}
Here, $\bfS = \diag(\bfs)$, $\bfLambda_\i = \diag(\bflambda_\i)$ and $\bfe \in \R^{m_i}$ is a vector of all ones,
and $\mu = \frac{\bfs^\top \bflambda_\i}{m_i}$ is a complementarity measure (it is zero at a KKT point). In interior point methods, the goal is to iteratively approximate a root of $F$ using Newton's method and some line search that ensures strict positivity of the slack and the Lagrange multiplier associated with the inequality constraints (i.e., $\bfs > \bfzero$ and $\bflambda_\i > \bfzero$).

At the $k$-th iteration of the interior point method, we denote the current iterates as $(\bff^k, \bflambda_\e^k,  \bfs^k, \bflambda_\i^k)$, linearize the optimality conditions~\eqref{eq:KKTCond}, and solve the linear system
\begin{equation}\label{eq:KKTnewton}
	\left[
	\begin{array}{c|c|c|c}
		\vdots & \vdots & \vdots & \vdots \\
		\nabla_{\bff} F & 	\nabla_{\bflambda_\e} F & \nabla_{\bfs} F & \nabla_{\bflambda_\i} F\\
		\vdots & \vdots & \vdots & \vdots
	\end{array}
	\right]	\left[
	\begin{array}{c}
		\Delta \bff \\ \Delta \bflambda_\e \\ \Delta \bfs \\ \Delta \bflambda_\i
	\end{array}
	\right]
	=
	\left[
	\begin{array}{c}
		- \bfr_d(\bfp)\\
		- \bfr_\e \\
		- \bfr_\i \\
		- \bfS \bfLambda_\i \bfe + \delta \mu \bfe
	\end{array}
	\right],
\end{equation}
where
\begin{equation}\label{eq:gradF}
	\left[
	\begin{array}{c|c|c|c}
		\vdots & \vdots & \vdots & \vdots \\
		\nabla_{\bff} F & 	\nabla_{\bflambda_\e} F & \nabla_{\bfs} F & \nabla_{\bflambda_\i} F\\
		\vdots & \vdots & \vdots & \vdots
	\end{array}
	\right]
   =
	\left[
		\begin{array}{cccc}
			\bfQ(\bfp) &  -\bfC_\e\t & \bfzero         & -\bfC_\i\t \\
			\bfC_\e     &  \bfzero       &  \bfzero       &      \bfzero  \\
			\bfC_\i     & \bfzero        & -\bfI_{m_i}     & \bfzero       \\
			\bfzero          &  \bfzero       & \bfLambda_\i & \bfS
		\end{array}
	\right].
\end{equation}
Here, the dual residual is
\begin{equation*}
	\bfr_d(\bfp) = \bfQ(\bfp) \bff^k + \bfb(\bfp) - \bfC_\e^\top \bflambda_\e^k - \bfC_\i^\top \bflambda_\i^k
\end{equation*}
and the primal residuals for the equality and inequality constraints are, respectively,
\begin{equation*}
	 \bfr_\e = \bfC_\e \bff^k - \bfc_\e\quad  \text{ and }\quad \bfr_\i = \bfC_\i \bff^k - \bfc_\i - \bfs^k.
\end{equation*}

The crucial components of the primal dual interior point method include an efficient linear solver for~\eqref{eq:KKTnewton}, the step length selection (here we use the largest step size in $[0,1)$ that ensures that both $\bflambda_\i$ and $\bfs$ remain sufficiently far from the boundary), and the choice of the centrality parameter $\delta$. For the latter, we use Mehrotra's predictor-corrector approach as described in~\cite{NocedalWright2006}. First, in the predictor step, we compute an \emph{affine scaling step} (i.e., we solve~\eqref{eq:KKTnewton} for $\delta = 0$) and perform a line search. Then, in the corrector step, we compute the final direction by solving~\eqref{eq:KKTnewton} for
\begin{equation*}
	\delta = \left( \frac{(\bff + \alpha^{\rm aff} \Delta \bff^{\rm aff})^\top (\bfs + \alpha^{\rm aff} \Delta \bfs^{\rm aff})}{m_e \mu }\right)^3.
\end{equation*}
Thus, each iteration requires two linear solves.

\paragraph{Computing sensitivities}
In order to enable fast optimization of the design parameters (i.e., optimization for the outer problem), we need to differentiate the solutions of the quadratic program~\eqref{eq:QPslack} with respect to the design parameters. To do this, we use implicit differentiation of the optimality condition~\eqref{eq:KKTCond}. Let $\widehat\bff(\bfp)$ be the computed solution to the quadratic programming problem. Then, we are interested in computing the sensitivity matrix $\bfJ_{\widehat\bff} (\bfp) \in \R^{n \times \ell}$ such that
\begin{equation*}
	\widehat \bff(\bfp+ \Delta \bfp) = \widehat \bff(\bfp) + \bfJ_{\widehat \bff}(\bfp) \Delta \bfp + \mathcal{O}(\| \Delta \bfp\|^2)
\end{equation*}
for all $\Delta \bfp \in \R^\ell$.

To this end, we differentiate both sides of \eqref{eq:KKTCond} around the current KKT point $(\widehat\bff(\bfp),\bflambda_\e(\bfp),\bfs(\bfp), \bflambda_\i(\bfp))$ with respect to $\bfp$ and obtain
\begin{align*}
	\bf0 & = \nabla_\bfp F(\widehat\bff(\bfp), \bflambda_\e(\bfp),\bfs(\bfp),\bflambda_\i(\bfp),\delta,\mu)\\
	& = \nabla_{\bfp} F(\widehat\bff(\bfp),\bflambda_\e, \bfs, \bflambda_\i,\delta, \mu) + \left[
	\begin{array}{c|c|c|c}
		\vdots & \vdots & \vdots & \vdots \\
		\nabla_{\widehat \bff} F & 	\nabla_{\bflambda_\e} F & \nabla_{\bfs} F & \nabla_{\bflambda_\i} F\\
		\vdots & \vdots & \vdots & \vdots
	\end{array}
	\right]	\left[
	\begin{array}{c}
		\bfJ_{\widehat\bff}(\bfp) \\ \bfJ_{\bflambda_\e}(\bfp) \\ \bfJ_{\bfs}(\bfp) \\ \bfJ_{\bflambda_\i}(\bfp)
	\end{array}
	\right].
\end{align*}
Assuming that the linear system in~\eqref{eq:gradF} is invertible (that is, there is a unique KKT point) we obtain
\begin{equation*}
\begin{bmatrix}
  	\bfJ_{\widehat\bff}(\bfp) \\ \bfJ_{\bflambda_\e}(\bfp) \\ \bfJ_{\bfs}(\bfp) \\ \bfJ_{\bflambda_\i}(\bfp)
    \end{bmatrix}
	  = -
    \begin{bmatrix}
			\bfQ(\bfp) &  -\bfC_\e\t & \bf0         & -\bfC_\i\t \\
			\bfC_\e     &  \bf0       &  \bf0        &      \bf0  \\
			\bfC_\i     & \bf0        & -\bfI_{m_i}     & \bf0       \\
			\bf0          &  \bf0       & \bfLambda_\i & \bfS
\end{bmatrix}^{-1}
	\begin{bmatrix} \nabla_{\bfp}(\bfQ(\bfp) \widehat \bff(\bfp) + \bfb(\bfp)) \\ \bf0 \\ \bf0 \\ \bf0
	\end{bmatrix}.
\end{equation*}
In the case of OED Problem A with $\bfE(\bfp) = \bfI_\ell$ we have
\begin{equation*}
	\nabla_{\bfp}(\bfQ(\bfp) \bff + \bfb(\bfp)) = 2 \begin{bmatrix} \bfR(\theta_1)\t \bfT\t &  & \\
  & \ddots & \\ & &  \bfR(\theta_\ell)\t \bfT\t
  \end{bmatrix} (\bfM(\bfp)\bff - \bfd(\bfp)).
\end{equation*}

Computing the complete sensitivity matrix of interest $\bfJ_{\widehat\bff}$ requires $\ell$ linear solves and thus might be infeasible when $\ell \gg 1$. Thus, as is also common in PDE parameter estimation~\cite{Haber2015}, we  provide matrix-free implementations that compute matrix vector products  $\bfv \mapsto \bfJ_{\widehat\bff}\bfv $ and  $\bfw \mapsto \bfJ_{\widehat \bff}^\top \bfw $ at the cost of one linear solve per training sample.

For some applications and in particular for OED Problem B, it may be beneficial to re-parameterize the angles contained in $\bfp$ by taking a reference angle $\delta_ 1 = p_1$ and non-negative increments $\delta_j\geq 0 $, $j = 2,\dots, n$ such that $p_k = \sum_{j = 1}^k \delta_j$.  This results in the following additional constraints
\begin{equation*}
	\delta_1 -a \geq 0, \quad
	\delta_j    \geq 0  \qquad \text{for } j = 2,\dots, n, \quad \mbox{ and } \quad
	b - \sum_{j = 1}^n \delta_j \geq 0,
\end{equation*}
where $a$ is the lower and $b$ the upper bound for the angles and derivatives are given by $$
\frac{\partial \bfp}{\partial \bfdelta} =
\begin{bmatrix}
	1      & 0 &      0 & \ldots & 0 \\
	1      & 1 &      0 & \ldots & 0 \\
	\vdots &   & \ddots &        & \vdots \\
	1      & 1 &      1 & \ldots & 1 \\
\end{bmatrix}\,.
$$

\section{Numerical Experiments} 
\label{sec:numerical_experiments}
In this section, we provide various examples for both OED problems A and B, using the tomography reconstruction problems described in Section~\ref{sec:problem_setup}.   We consider various assumptions (including different training data sets) and investigate how different constraints on the imaging state problem may affect the optimal design parameters.

\subsection{Implementation} 
\label{sub:implementation}
Our numerical framework for minimizing the empirical Bayes risk is implemented as an add-on to the parameter estimation package \textsc{jInv}~\cite{RuthottoTreisterHaber2017}. To benefit from \textsc{jInv}'s existing capabilities, the routines for solving the lower-level problem and computing sensitivities are implemented as a module, extending the abstract forward problem type. In particular, this allows parallel and distributed-memory evaluation of the constrained inverse problems for different training data. The design problem is formulated and solved using the misfit, regularization, and optimization methods provided in \textsc{jInv}'s. Our module will be made freely available. 


\subsection{Data sets and Constraints} 
\label{sub:trainingdata}
We consider four sets of training data shown in Figure~\ref{fig:testData} and described below.
To obtain each projection, we use model~\eqref{eq:forwardmodel} with parallel rays where the number of rays is equal to the pixel size.

\begin{description}
	\item[Rectangles:] The data set consists of $20$ binary images of size $40\times 40$ that show randomly spaced and sized rectangles. The rectangles' edges are aligned with the coordinate axes. Note that given the knowledge on the binary constraints, exact reconstruction is possible using two measurements with angles of $0$ and $90$ degrees, respectively.

	\item[Pentagons:]
	A second data set consists of binary images of pentagons. As before, the training data consists of $20$ discrete images with $40\times 40$ pixels. The location and the size of the pentagons change randomly from image to image; however, the angles between the edges and the coordinate axes are the same. Thus, the object can be reconstructed exactly with appropriate constraints and from five projection angles.

	\item[Shapes:]
	As a non-binary example we consider a synthetically generated dataset containing $20$ discrete images of size $40\times 40$ that are obtained by evaluating a random smooth function on a randomly chosen supporting set.

	\item[Phantom:] As a more realistic example, we generate a training data set consisting of $20$ gray valued images resembling random variations of the Modified Shepp-Logan phantom \cite{jain1989fundamentals,shepp1974fourier}. The Shepp-Logan phantom resembles basic head characteristics: exterior, skull, left and right ventricles, as well as tumors. We randomly varied head features, such as skull size, head \& ventricle size and orientation, intensity, and number of tumors. Our Matlab implementation will be publicly available.  Here, we discretize the data on a $64\times64$ regular grid.
\end{description}
\begin{figure}[bthp]
	\begin{center}
		\iwidth=50mm
		\begin{tabular}{@{}cc@{}}
			\includegraphics[width=\iwidth]{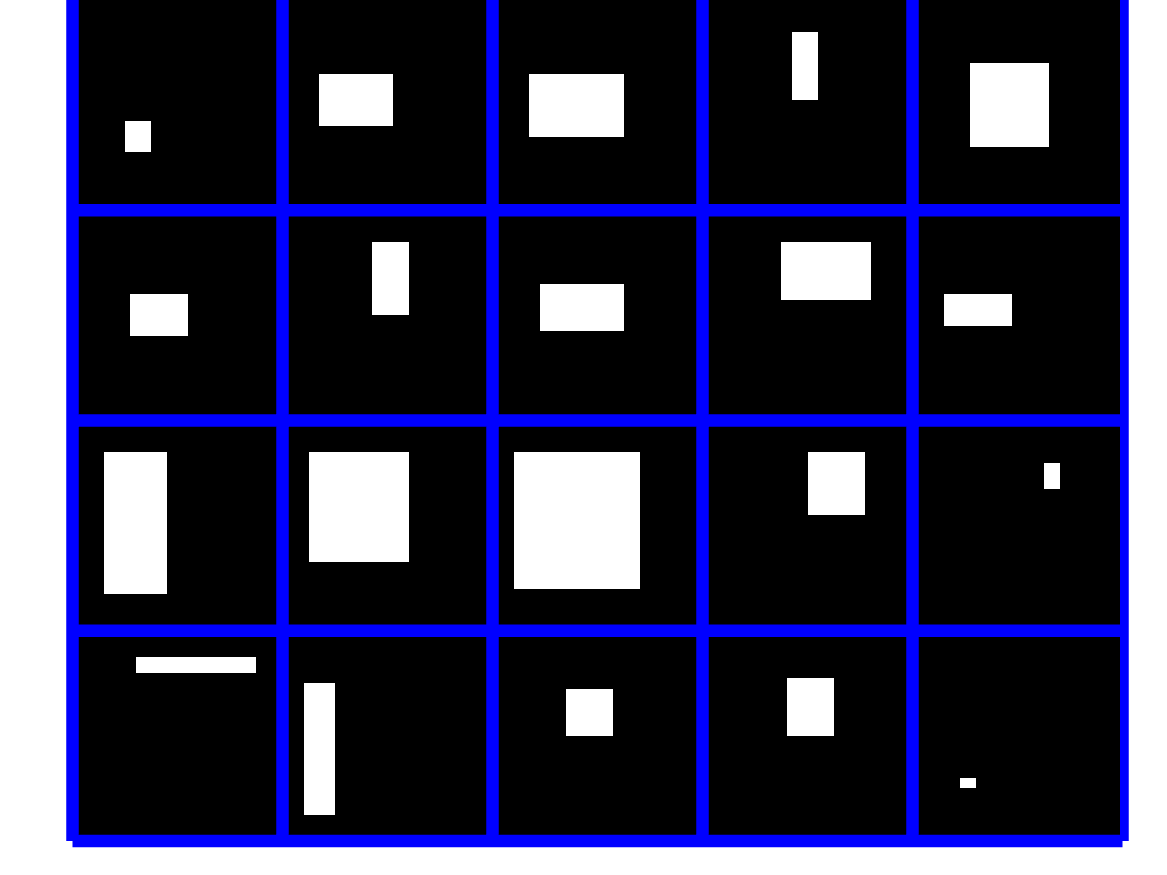}
			&
			\includegraphics[width=\iwidth]{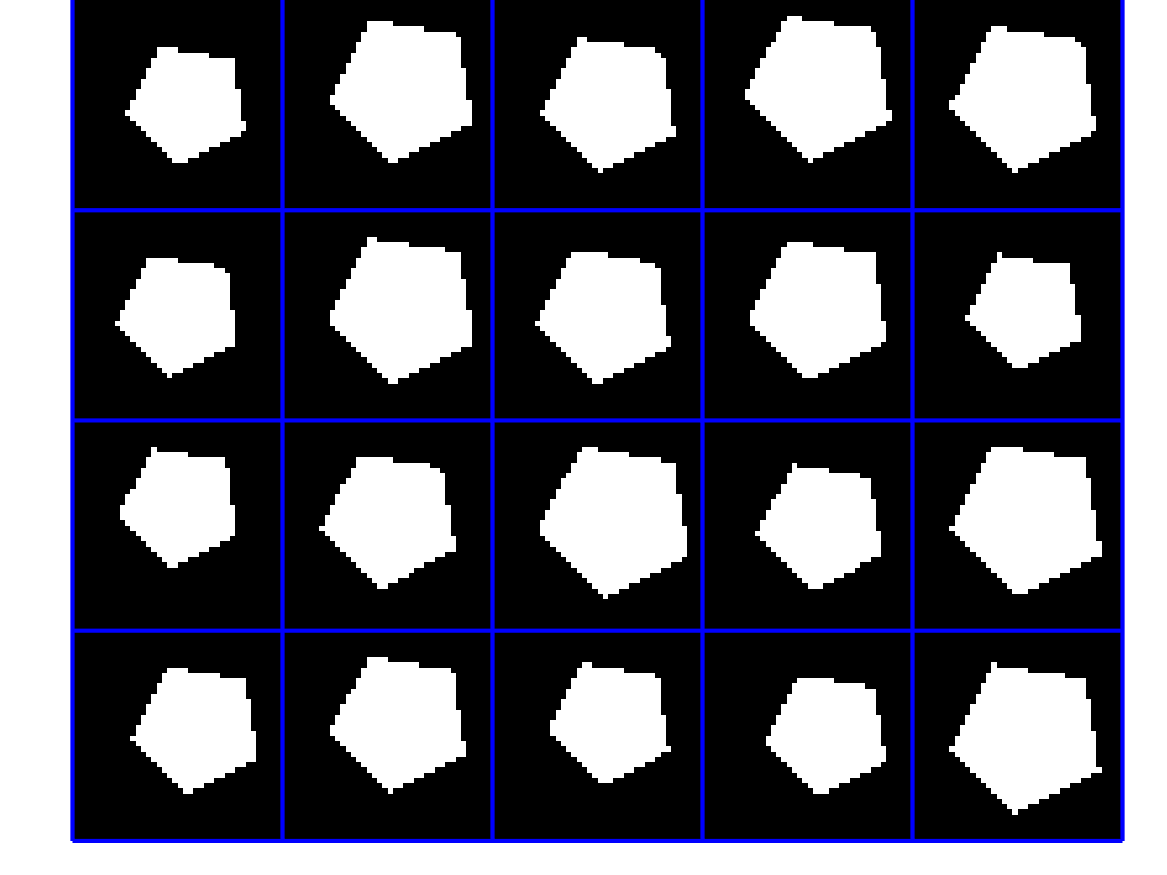}\\
			1) rectangles
			&
			2) pentagons\\[2ex]
			\includegraphics[width=\iwidth]{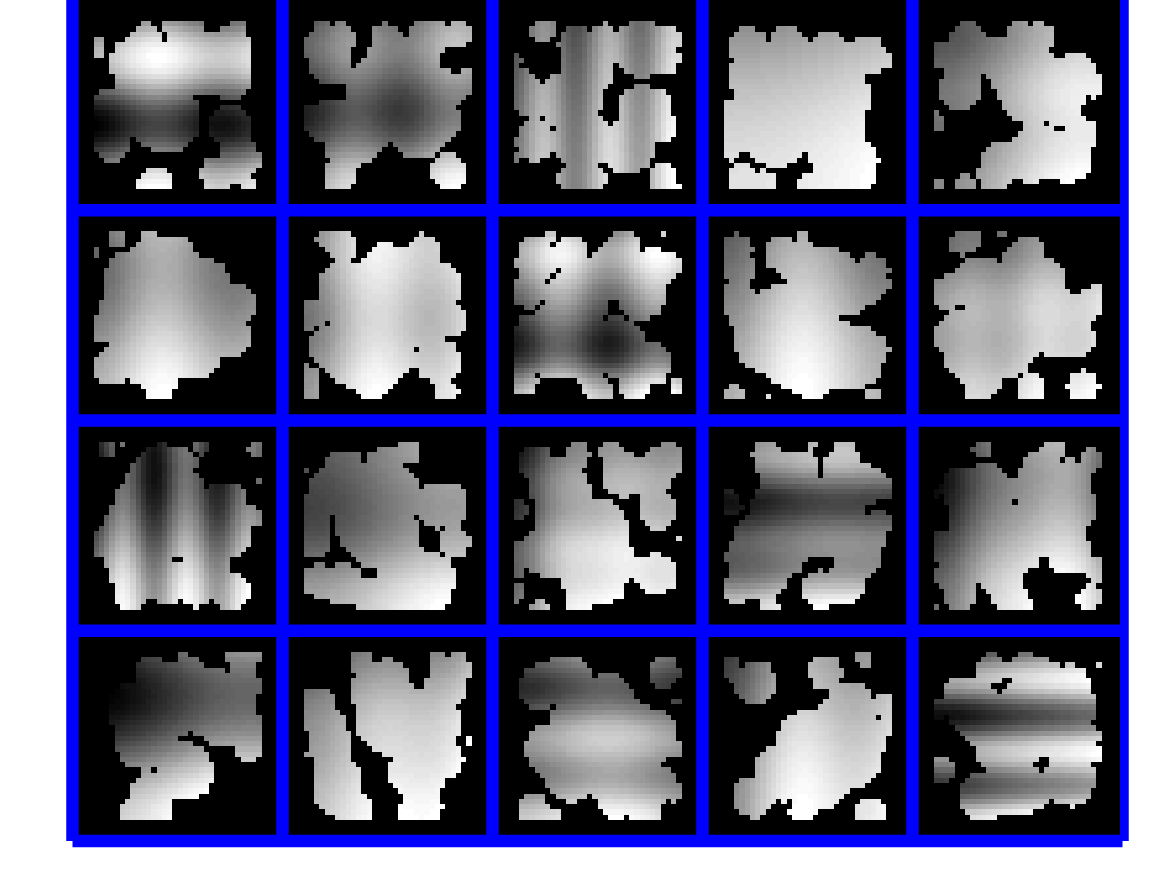}
			&
			\includegraphics[width=\iwidth]{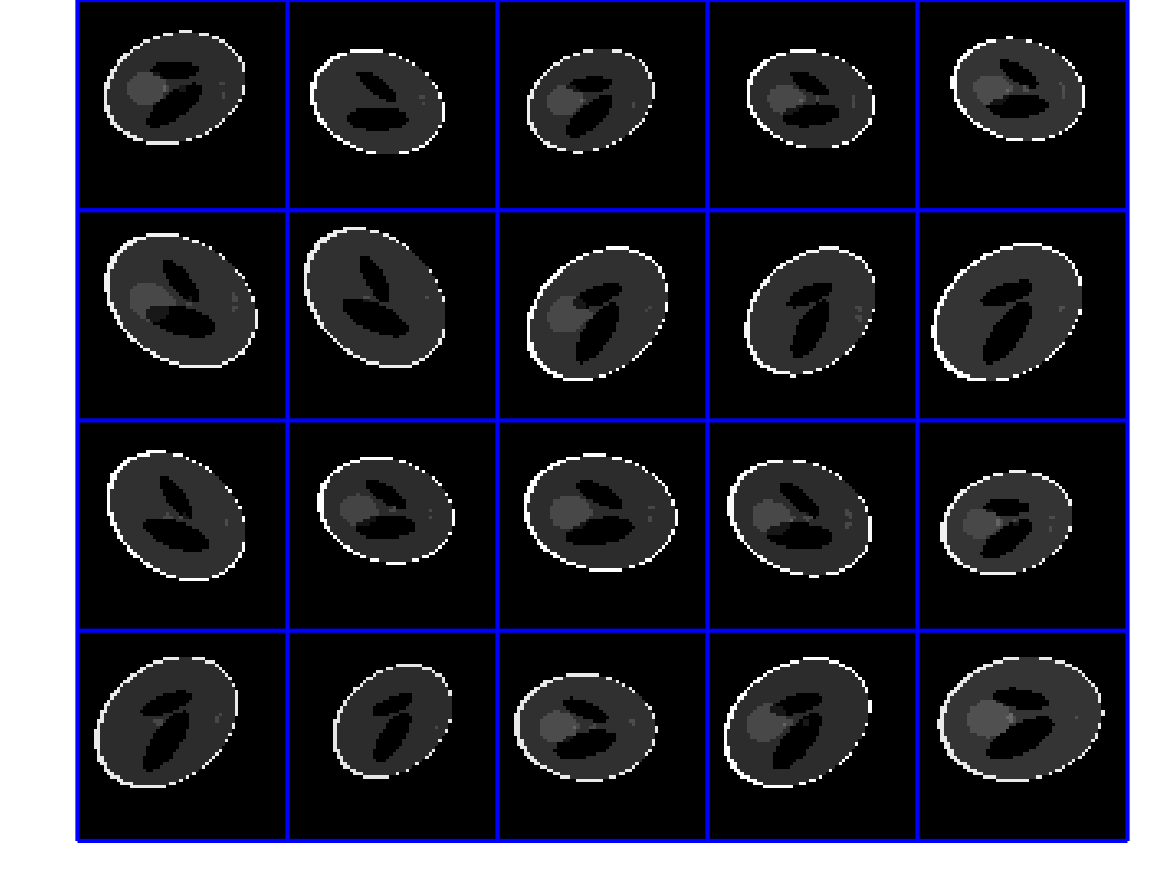}
			\\
			3) shapes
			&
			4) phantom
		\end{tabular}
	\end{center}
	\caption{Simulated training data used in the numerical experiments. Each data set consists of 20 images with intensity values between $0$ (black) and $1$ (white). The blue lines divide the plot into the individual examples.}
	\label{fig:testData}
\end{figure}
These data sets can be used directly in empirical Bayes risk minimization OED problems for approximating the expected value, as discussed in Section~\ref{sec:problemformulation}.  Furthermore, we can treat the images as samples from some underlying distributions and either use a prior assumption on the covariance $\bfGamma_\bff$ or a sample approximation to solve the Bayes OED problems.

The main interest of this work is to investigate how different constraints on the imaging problem may affect the optimal design parameters. Therefore, we consider four common choices of constraints in~\eqref{eq:ip}.
\begin{description}
	\item[unconstrained:] No constraints are imposed on the discrete image, $\bff$. In this case the optimality condition in~\eqref{eq:ip} is a linear system with a positive definite matrix and (for small-scale problems) is solved using a Cholesky factorization. For this case, we compare the optimal Bayes design with the optimal empirical Bayes design.

	\item[equality constraints:] We assume that the sum of the intensity values equals a known constant. To this end, we set $\bfC_\e = \begin{bmatrix} 1, & \ldots ,& 1\end{bmatrix} \in \R^{1 \times n}$ and $\bfc_\e = \bfC_\e \bff_{\rm true}$. Similar to the previous case the optimality condition~\eqref{eq:ip} is linear and requires solving a saddle point problem. To this end, we use an LU factorization.

	\item[non-negativity:] A physically meaningful constraint in many imaging problems is to enforce reconstructed intensity values to be non-negative. In our formulation we use $\bfC_\i = \bfI_n \in \R^{n\times n}$ and $\bfc_\i = \bfzero$.

	\item[bound constraints:] In addition to non-negativity, we enforce an upper bound and restrict the intensity values of each pixel to be between $0$ and $1$ by using $\bfC_\i = [\bfI_n; -\bfI_n] \in \bbR^{2n\times n}$ and $\bfc_\i = [\bfzero; -\bfe] \in \bbR^{2n}$.
\end{description}
Since the most plausible choice of a constraint will depend on the particular application our framework supports a variety of constraints.  Due to the relative simplicity of OED problem B (e.g., no additional design regularization and potentially much fewer parameters to optimize), we begin with some investigations on the impact of lower level constraints on the overall optimal experimental design.

\subsection{OED Results for Problem B} 
\label{sub:comparison_of_objective_functions}
For simplicity and for visualization purposes, we consider OED problem B with $\ell=2$; that is, we aim to find the $2$ projection angles, where the resulting reconstructions minimize the mean-squared error.  Note that computing reconstructions in this case is highly under-determined.

We first investigate the OED objective function $\calJ(\bfp)$ for projection angles $\bfp \in [0,180]^2$ in intervals of $1$ degree, assuming $p_1 \geq p_2$.
Note that no regularization is included in the outer optimization problem of OED problem B, and thus the objective function measures the mean squared error of the reconstruction.  We begin with no constraints on the inner problem and provide the Bayes risk values in Figure~\ref{fig:Bayes} for various covariance matrices where the underlying images are $64 \times 64$ pixels.  First, we provide the Bayes risk for $\bfL = \bfI_{4096}$, where the minimum occurs around angles $45$ and $135$ degrees.  Then for both the rectangles and phantoms data sets, we generated 1,000,000 sample images and computed the sample covariance matrices (with a small regularization to ensure positive definiteness).  Using the computed factor $\bfL$, we computed Bayes risks and provide them along with their corresponding minima in Figure~\ref{fig:Bayes}. In general, we see that $p_1=p_2$ as well as $p_1=0, p_2=180$ correspond to higher Bayes risk, which seems intuitive for tomography.  However, it is evident that the choice of $\bfL$ does affect the optimal design.

\begin{figure}[t]
	\begin{center}
		\begin{tabular}{@{}c@{}c@{}c@{}}
		$\bfL = \bfI$ & rectangles & phantom\\
\includegraphics[width=.33\textwidth]{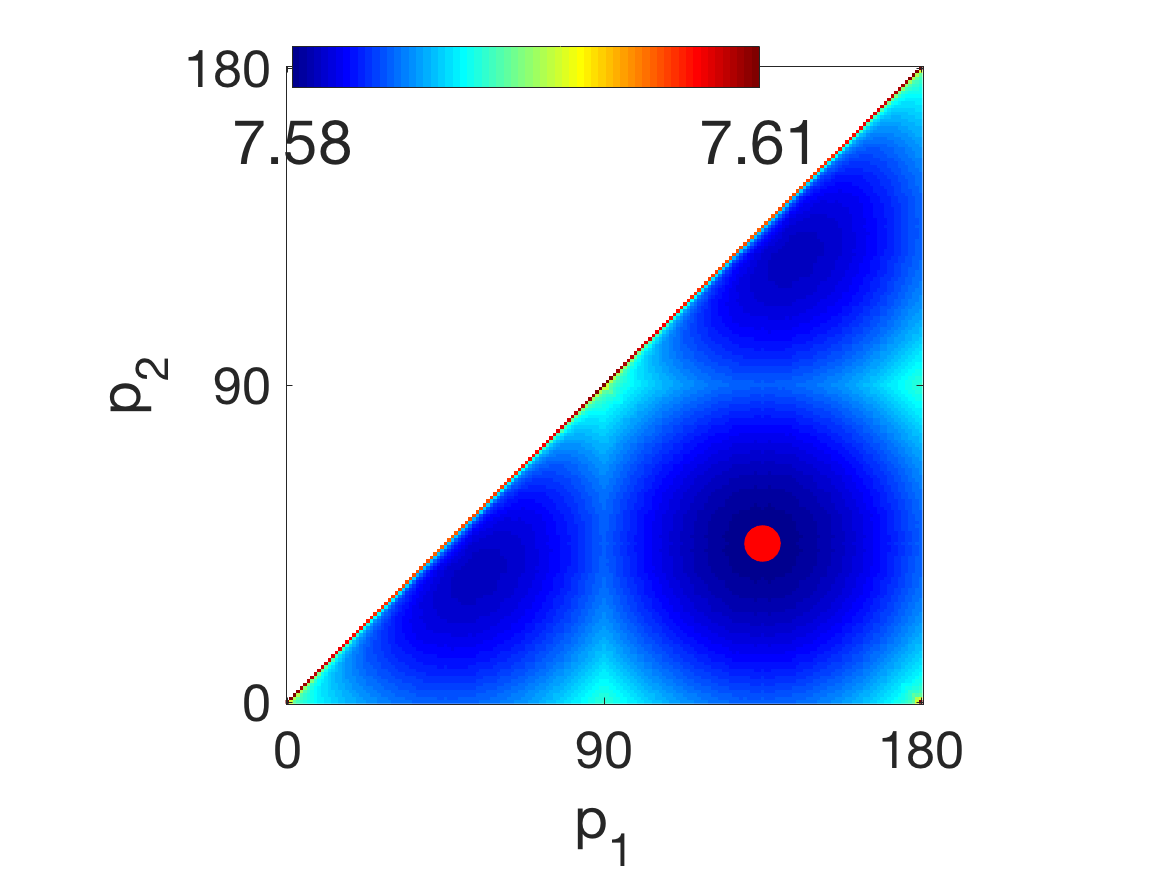}	&
\includegraphics[width=.33\textwidth]{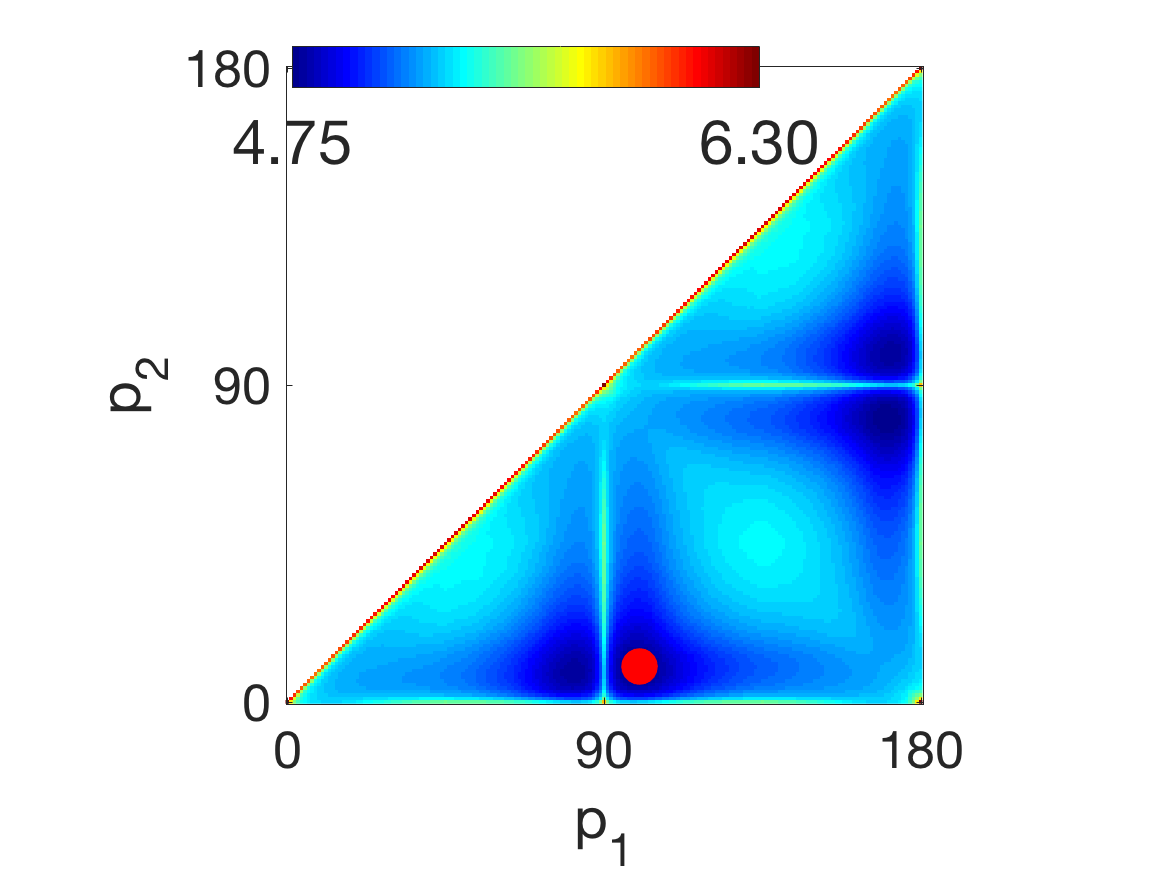}	&
\includegraphics[width=.33\textwidth]{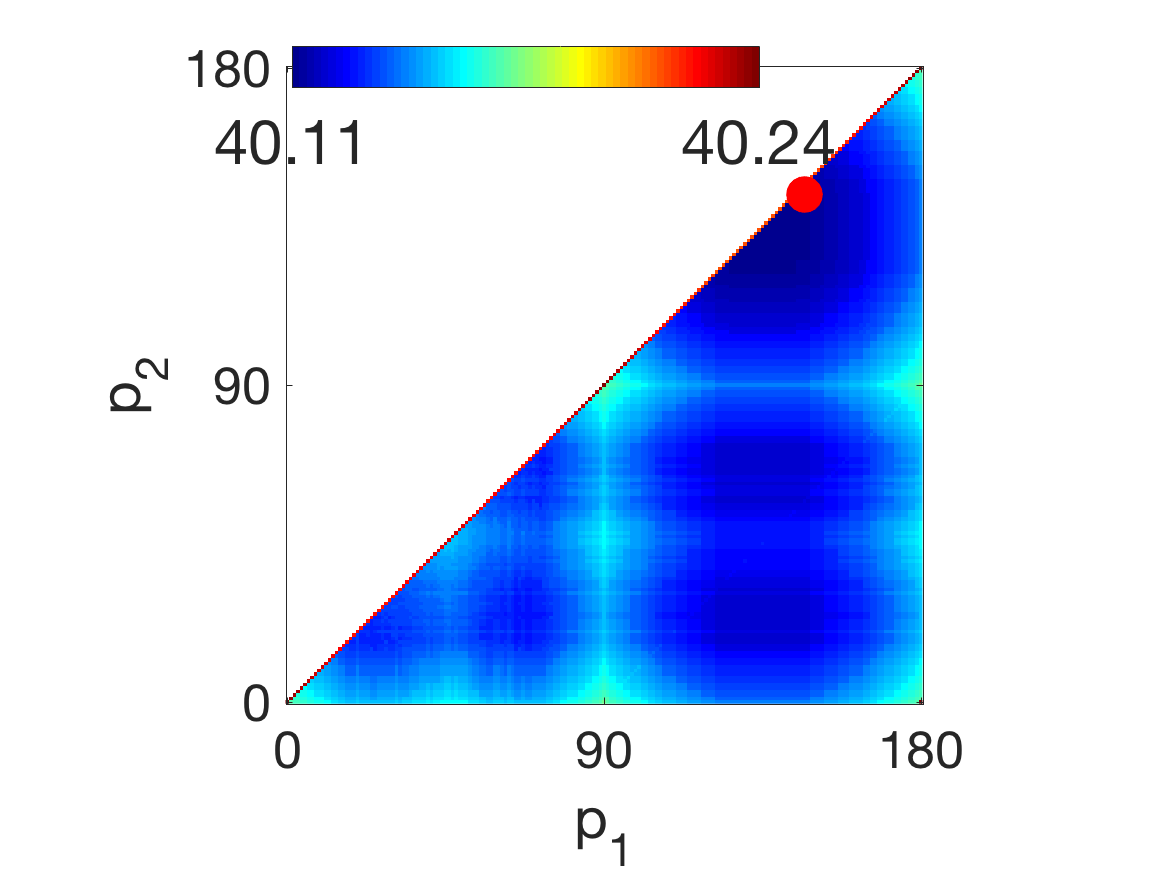}
		\end{tabular}
	\end{center}
	\caption{Bayes risk for OED Problem B with two projection angles $p_1$ and $p_2$ for $\bfL$ being the identity matrix and for $\bfL$ coming from covariance matrices obtained from realizations of the rectangles and the phantoms datasets.  Red dots correspond to minimum values. These results correspond to the unconstrained state problem with $\alpha =1$ and $\sigma = 1$. }
	\label{fig:Bayes}
\end{figure}

Next we provide the \emph{empirical} Bayes risk (i.e., values of the sampled objective function) for all four training data sets and four options for constraints in Figure~\ref{fig:OEDB_plotObjective}. The two best designs are marked with red dots. For all training data, the design improvement is most pronounced for the box-constrained lower-level problem. Further, for the pentagon example, the global optima are obtained at different points for unconstrained/equality constrained and the inequality constrained problems.
\begin{figure}[t]
	\renewcommand{\rottext}[1]{\rotatebox{90}{\hbox to 30mm{\hss #1\hss}}}
	\newcommand{\image}[1]{\includegraphics[width=31mm,trim=50 0 60 0,clip=true]{#1}}
	\begin{center}
		\begin{tabular}{@{}c@{}c@{}c@{}c@{}c@{}}
			& unconstrained & equality & non-negativity & box\\
		\rottext{rectangles}
		 & \image{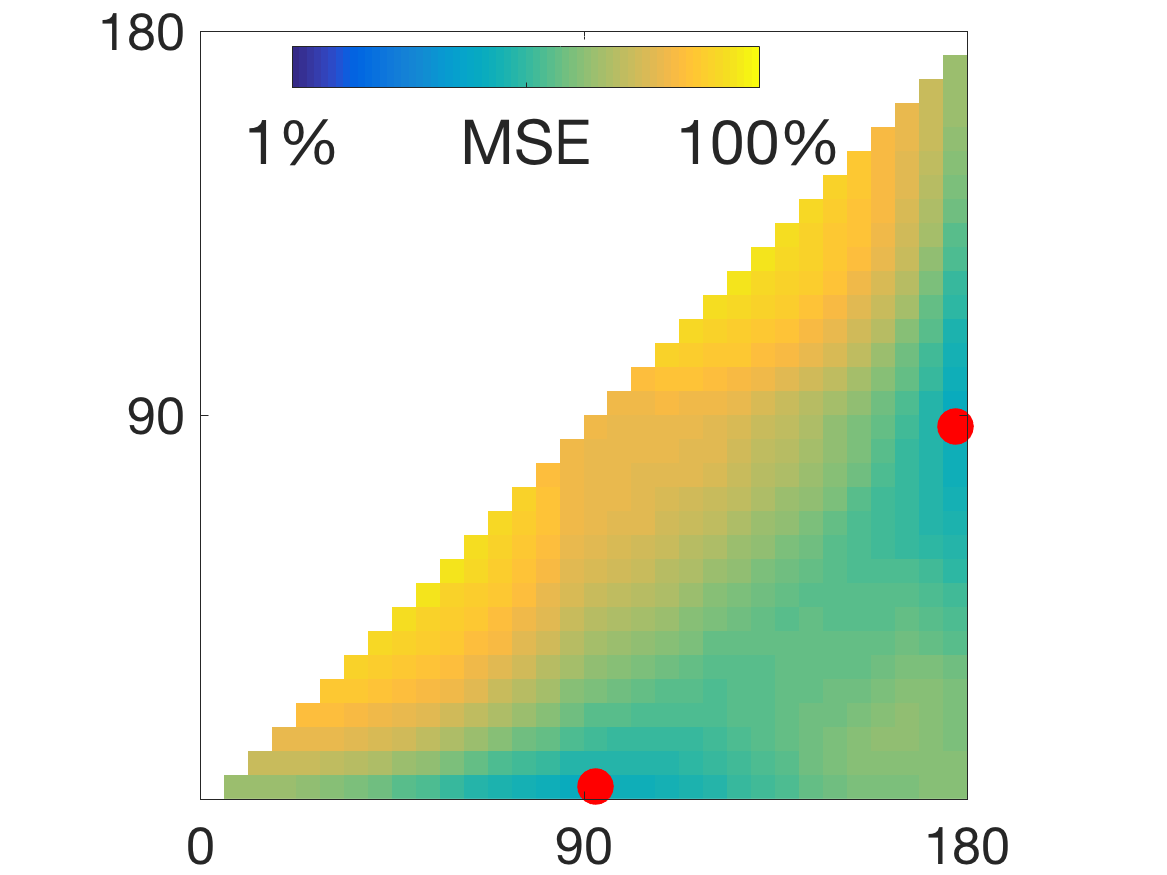}
		 & \image{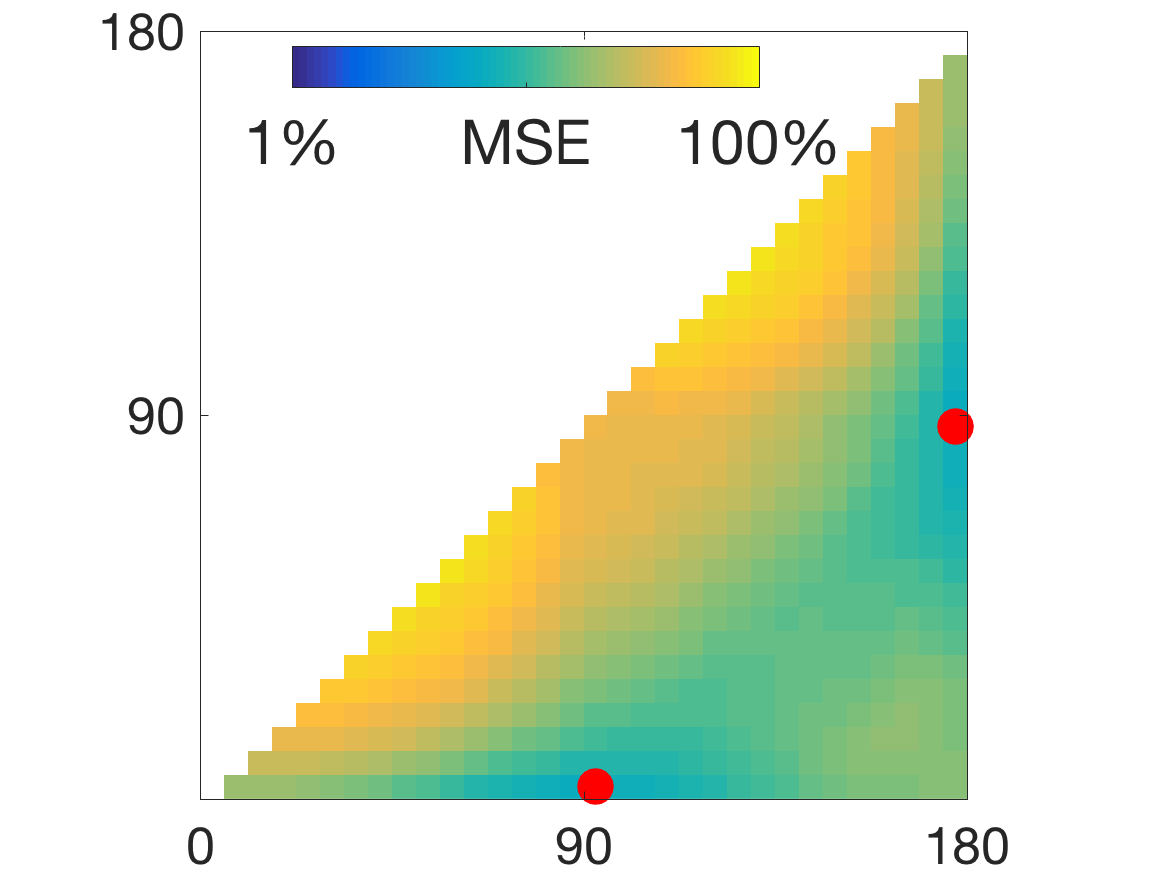}
		 & \image{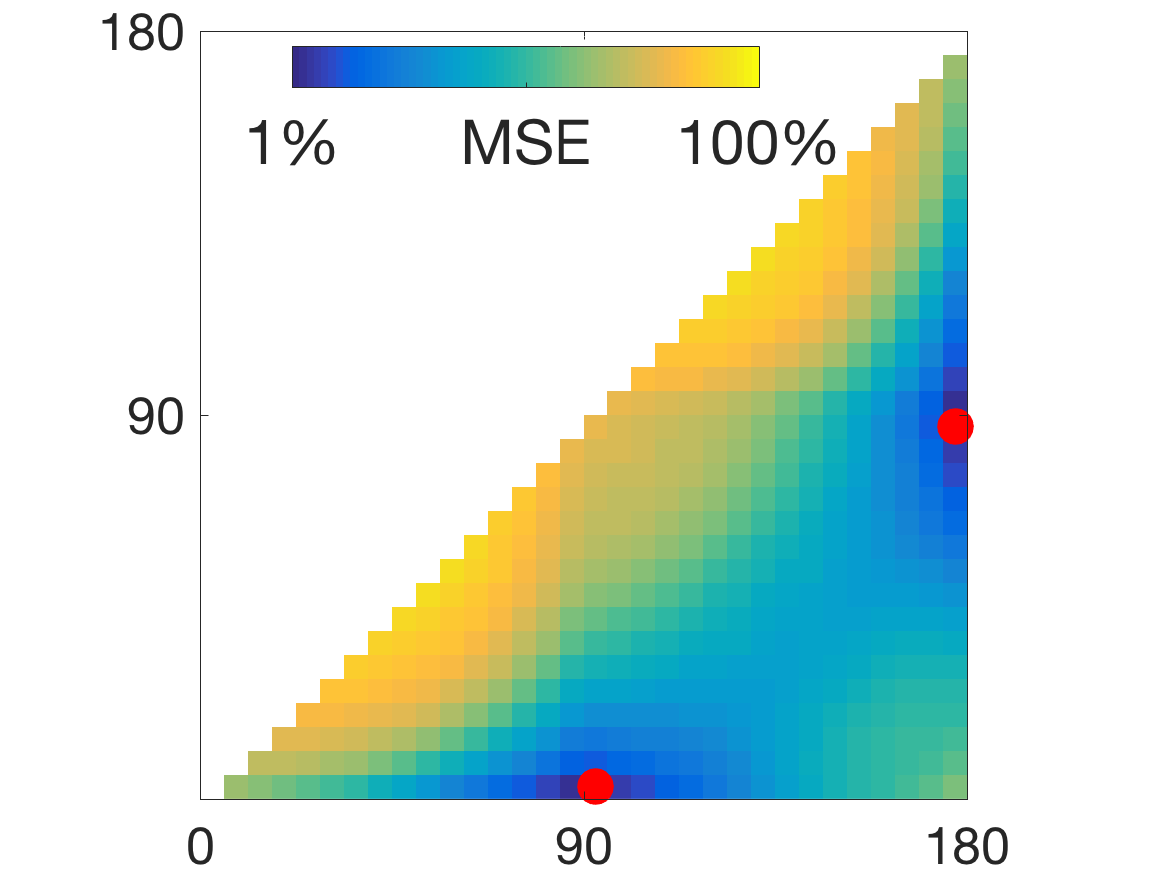}
		 & \image{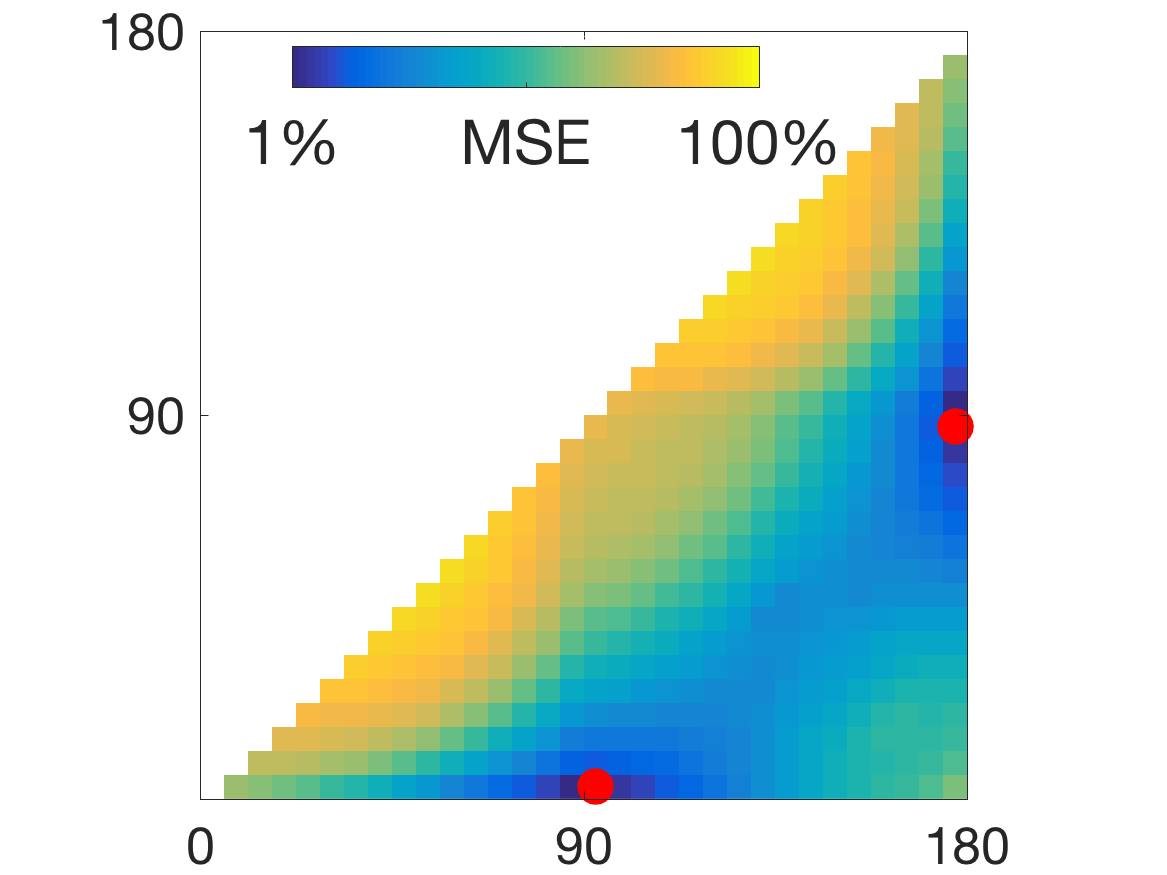}\\
		\rottext{pentagon}
		 & \image{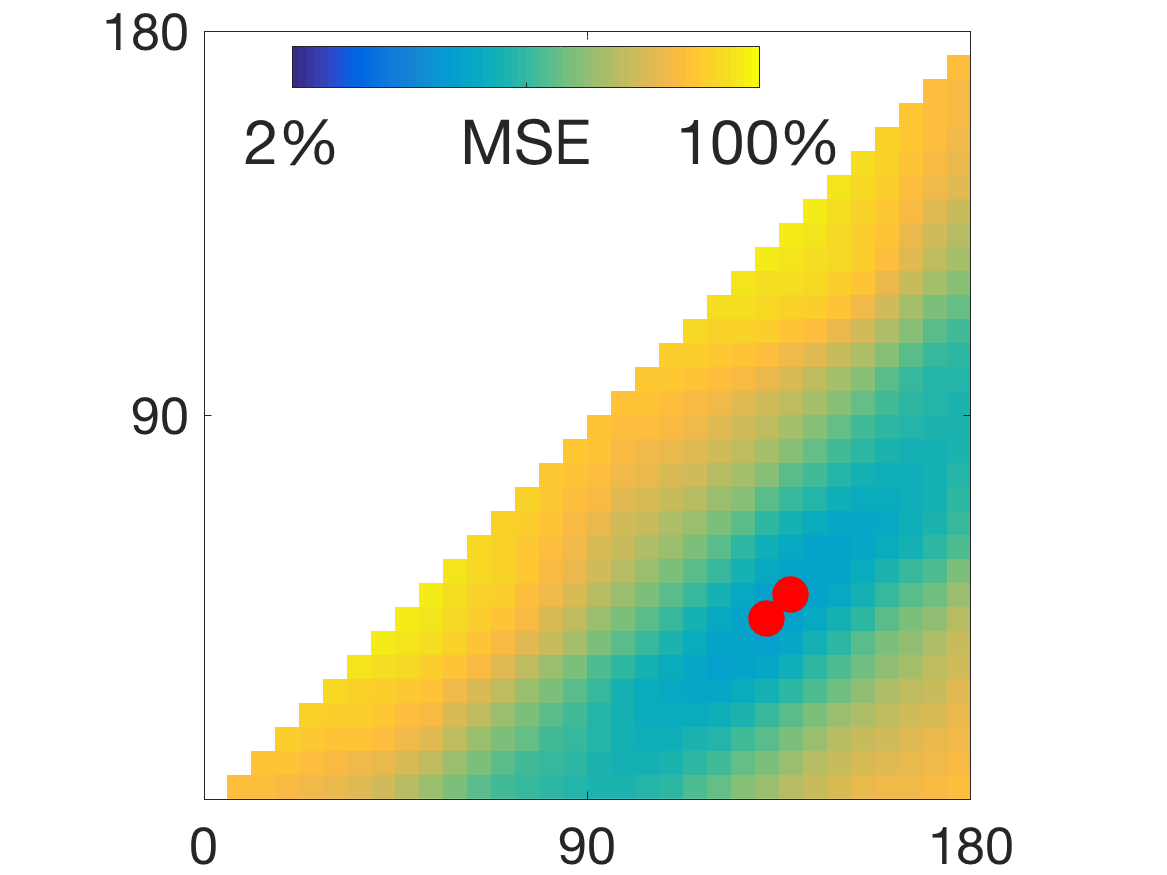}
		 & \image{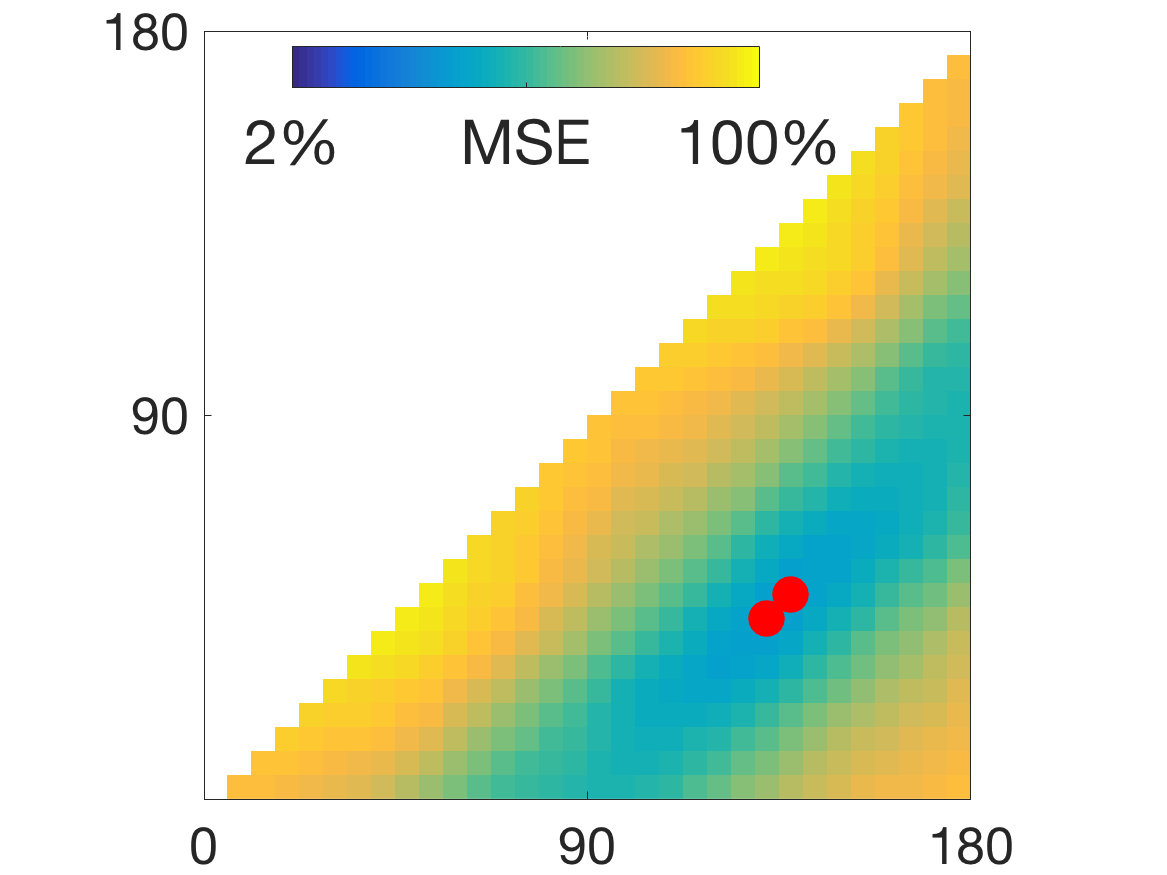}
		 & \image{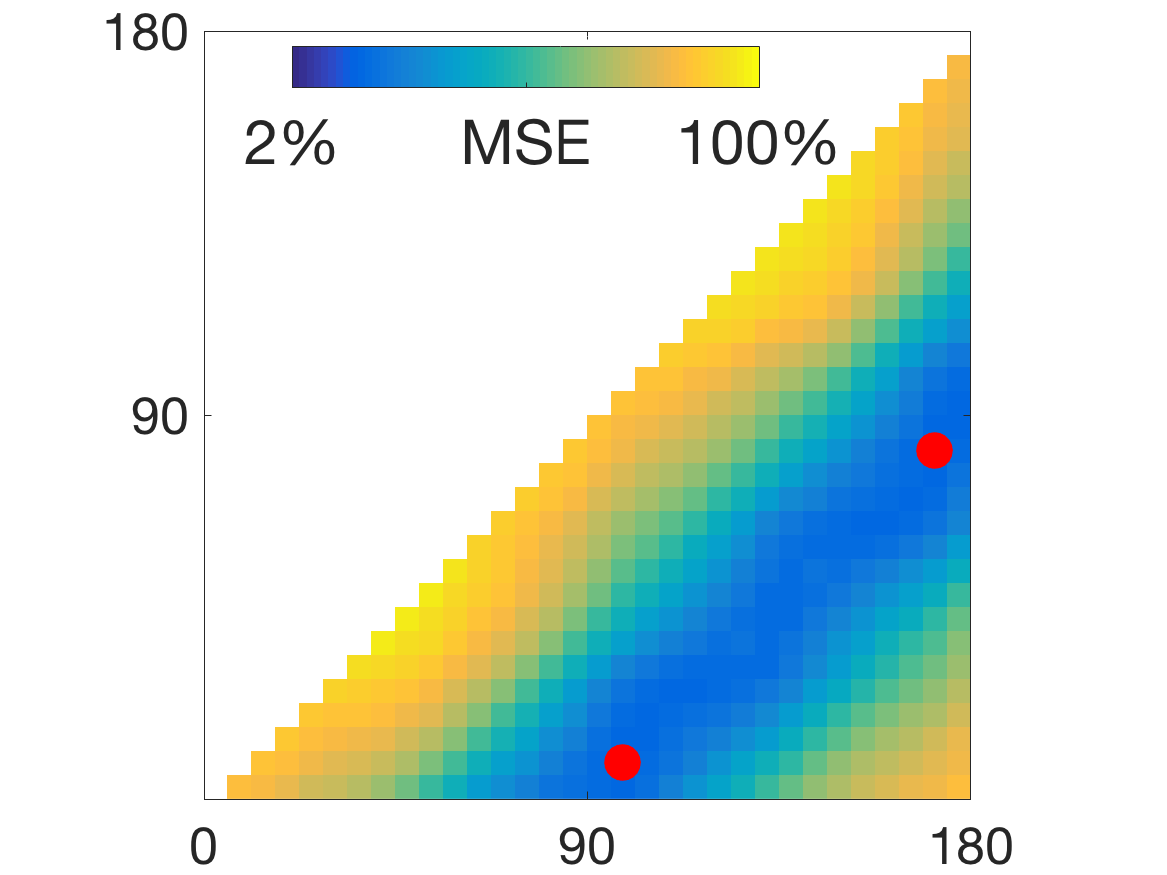}
		 & \image{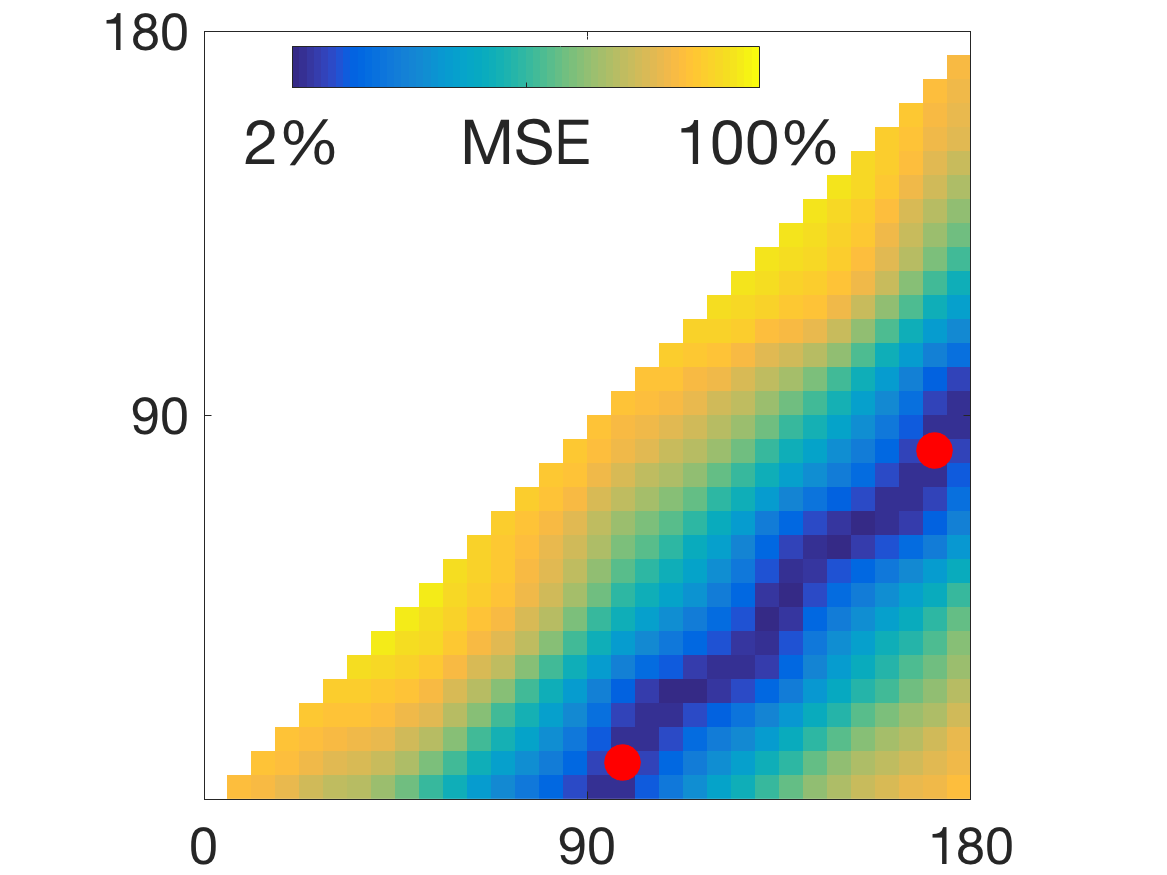}\\
		\rottext{shapes}
		 & \image{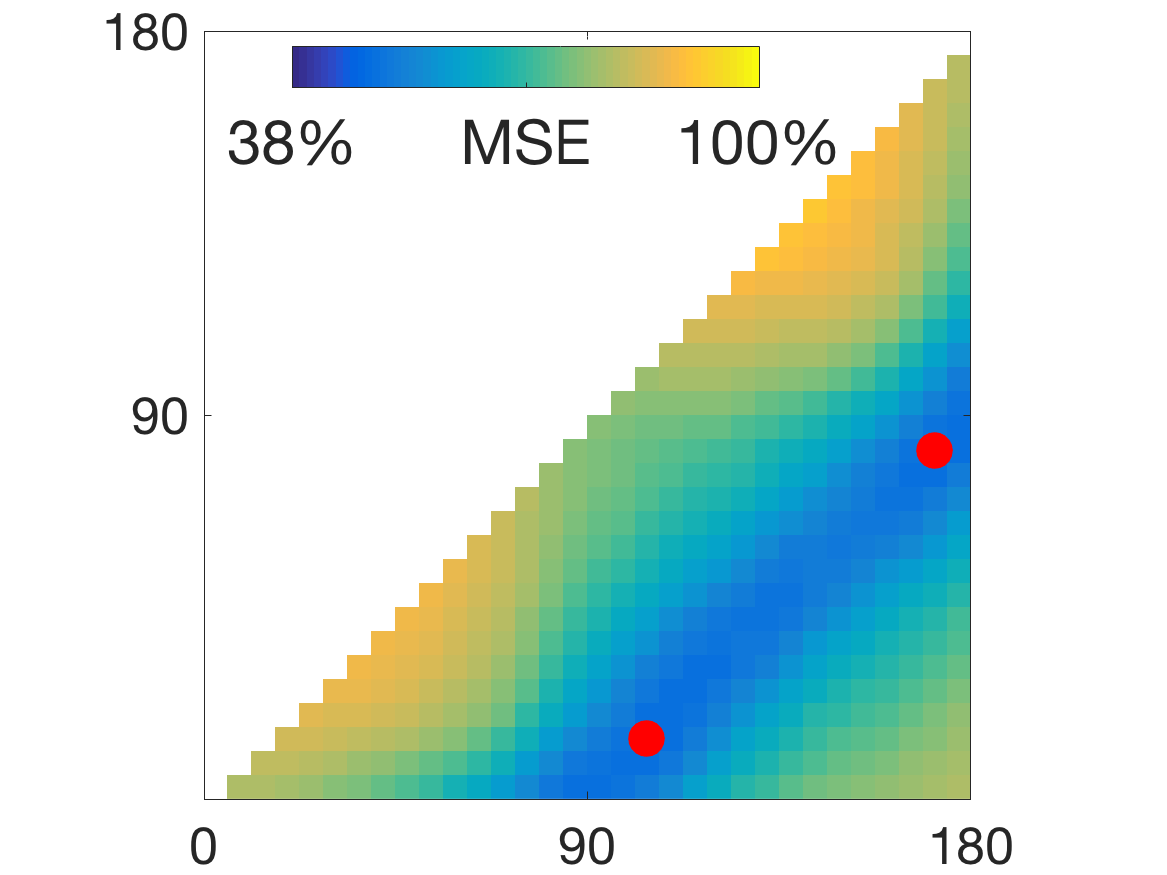}
		 & \image{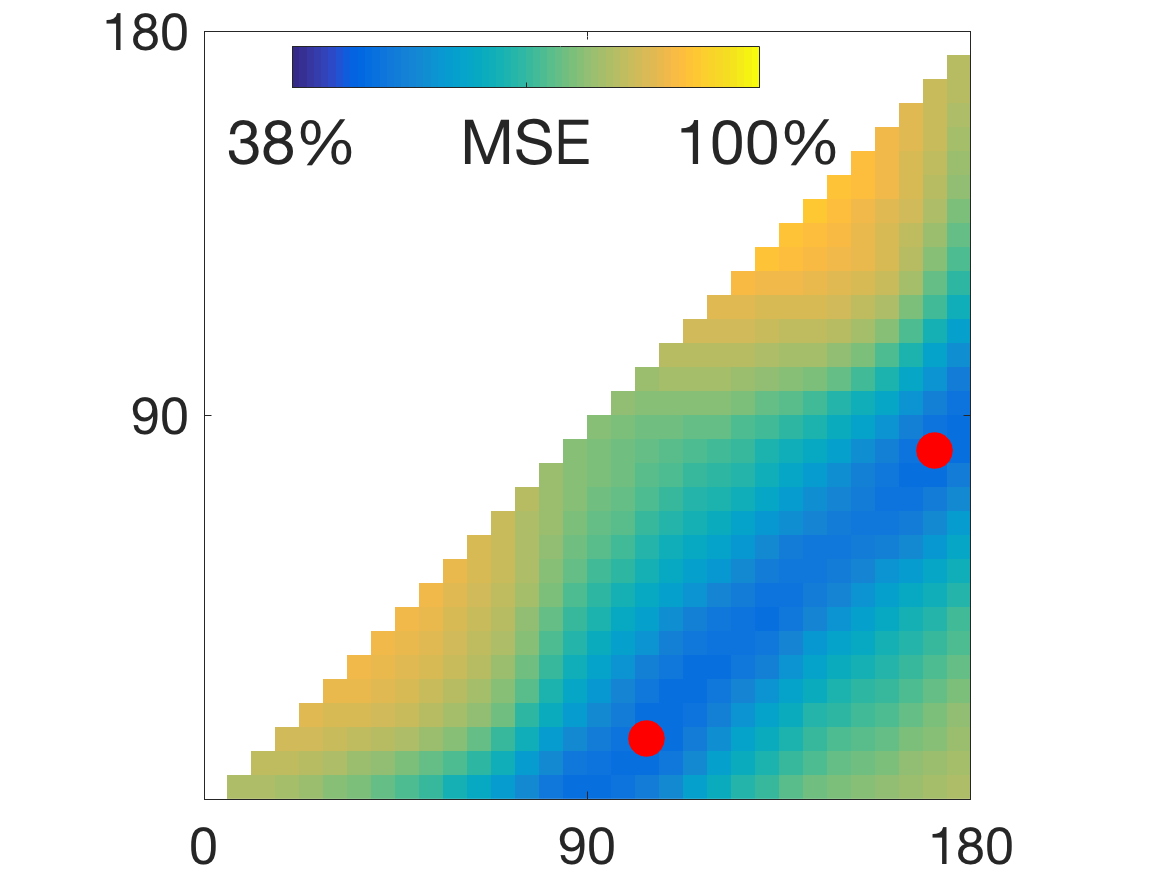}
		 & \image{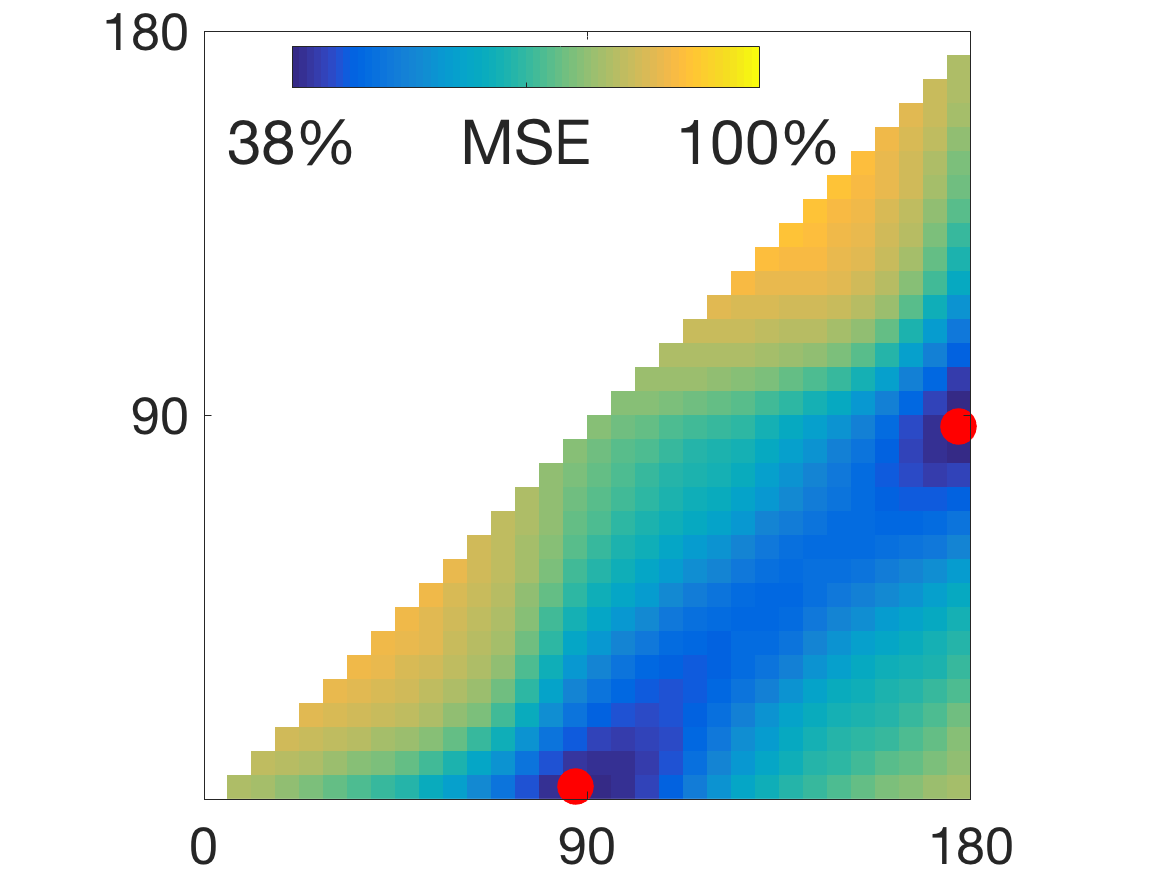}
		 & \image{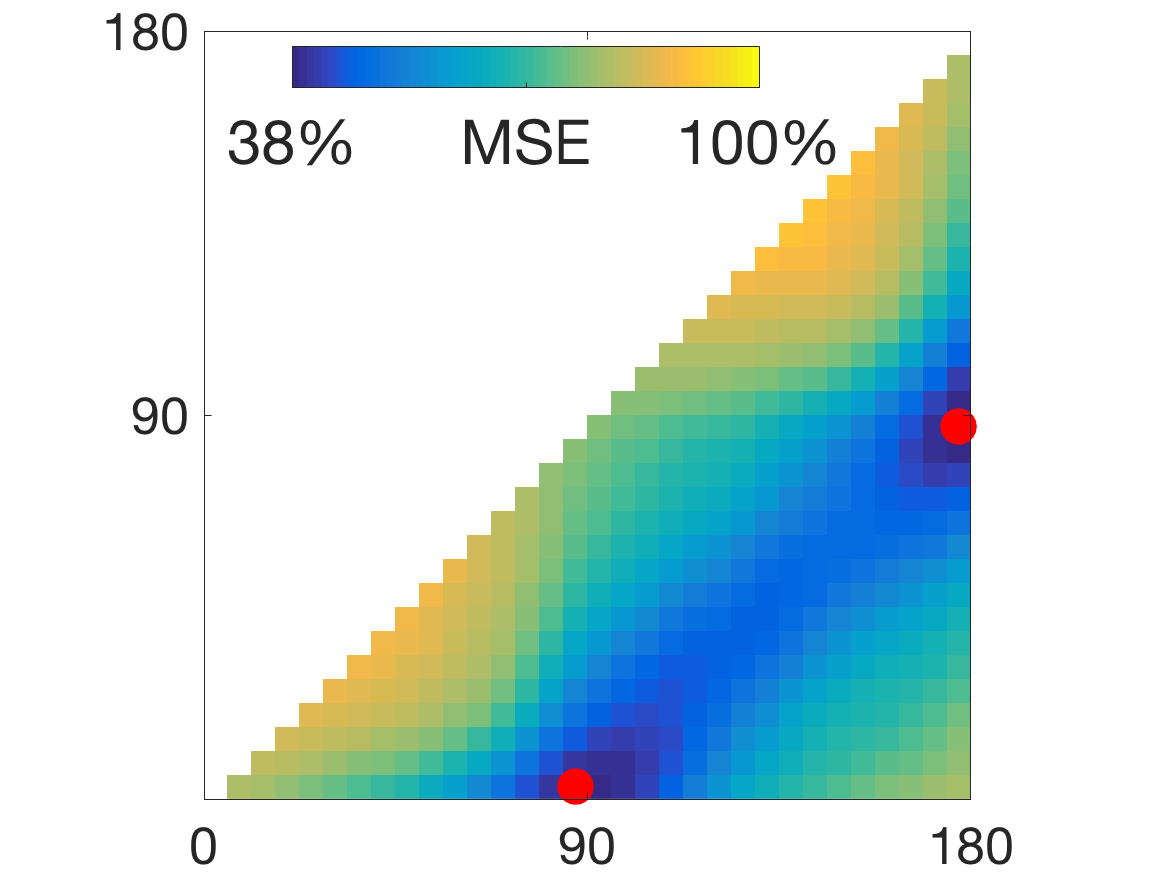}\\
		\rottext{phantom}
		 & \image{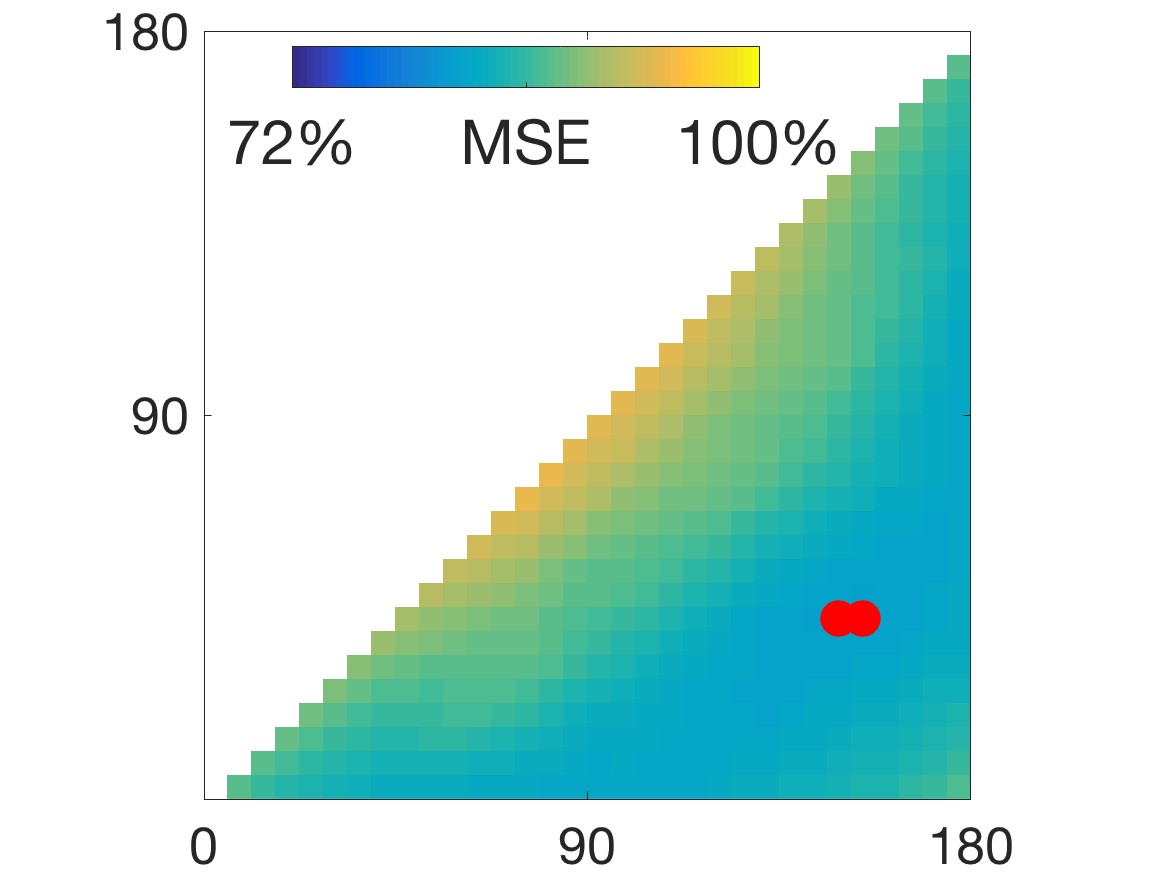}
		 & \image{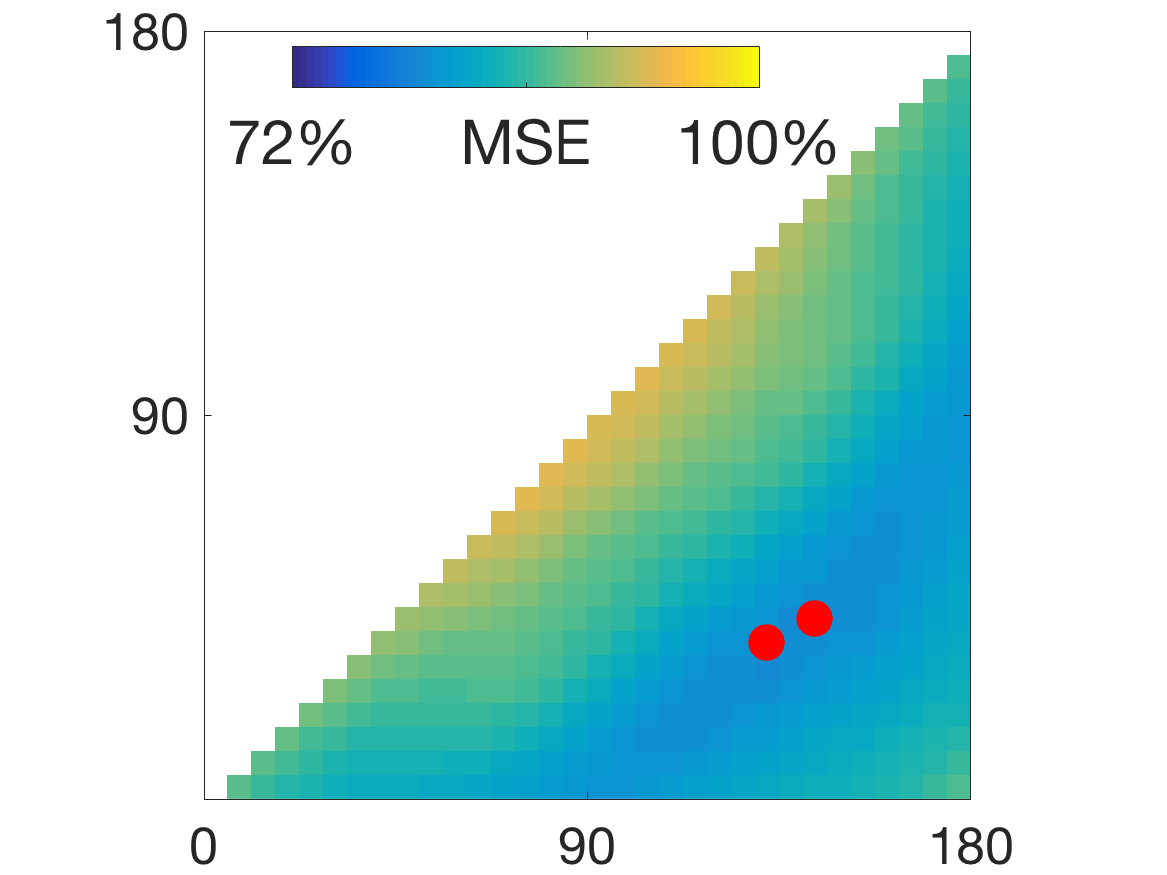}
		 & \image{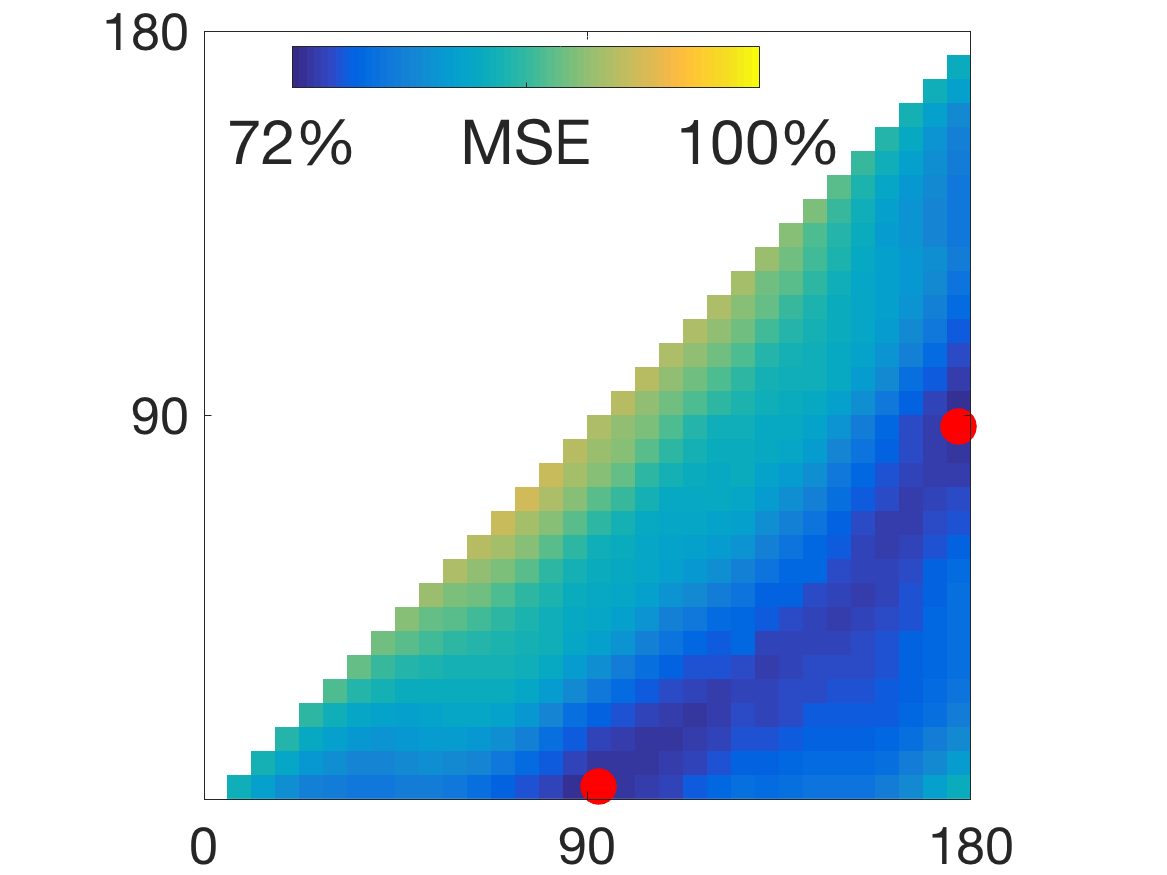}
		 & \image{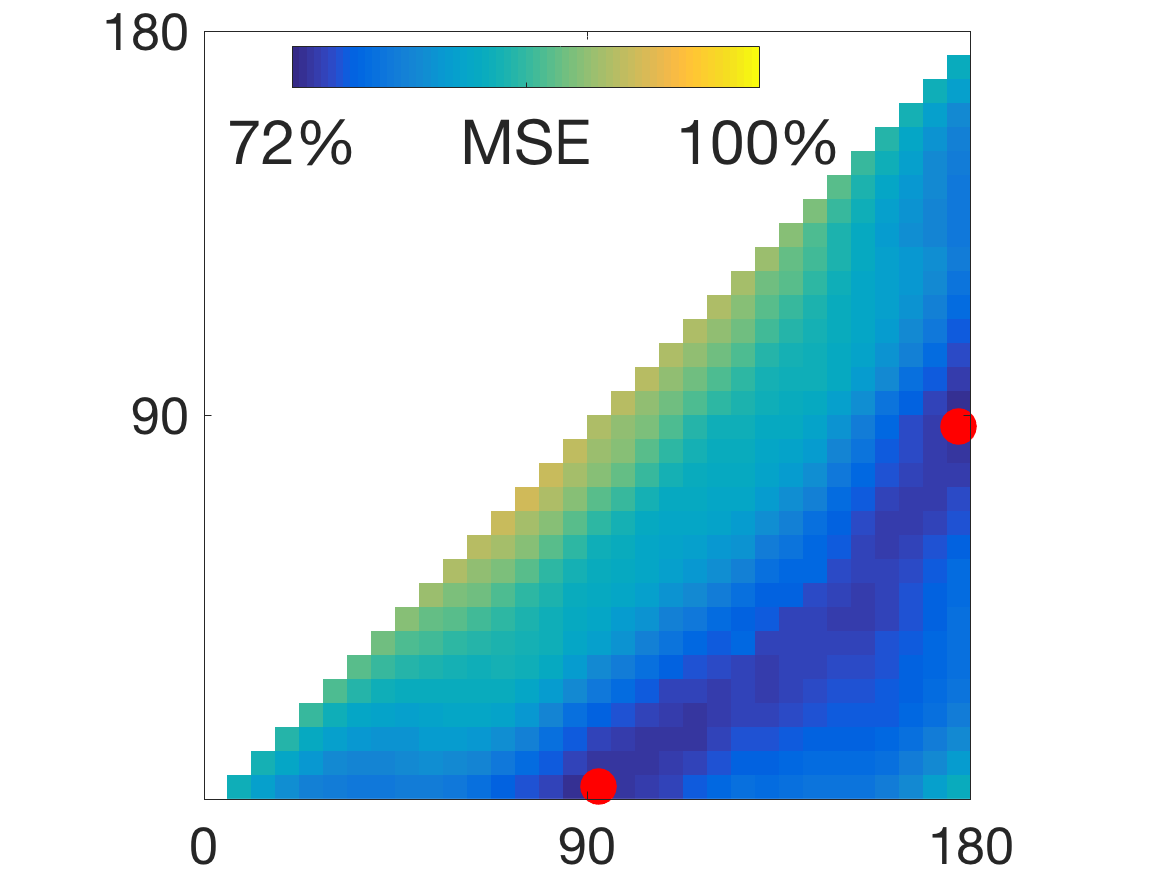}\\
		\end{tabular}
	\end{center}
	\caption{Mean squared reconstruction error for OED Problem B with two projection angles plotted along the $x$ and $y$ axis. For each data set we show the reduction of MSE for the unconstrained, equality constrained, non-negativity constrained and box constrained OED problem. The two best designs are indicated by red dots. To allow for comparison, the same color axis is in the sub plots of each row.  It can be seen that a considerably larger reduction of the reconstruction error is achieved for the constrained problems. }
	\label{fig:OEDB_plotObjective}
\end{figure}

We investigate the impact of the value of the regularization parameter, $\alpha$, on the reconstruction error for the optimal designs identified in the previous steps. To this end, we compute the MSE for 20 logarithmically equal spaced values of $\alpha$ between $10^{-4}$ and $10^3$ using the optimal projection angles determined in the previous step. We see that for the unconstrained and equality constrained problem, the MSE is less sensitive to the choice of $\alpha$ and thus the global minima are difficult to identify.  For the inequality constrained problems substantial improvement of the design can be obtained.  For all datasets, we found that thee box constraints resulted in smaller MSE values overall.  Since a good choice of $\alpha$ may not be available a priori, one could consider an approach that incorporates $\alpha$ as a design parameter, so that optimal design also includes optimizing for $\alpha$.

\begin{figure}[t]
	\begin{center}
		\includegraphics[width=\textwidth]{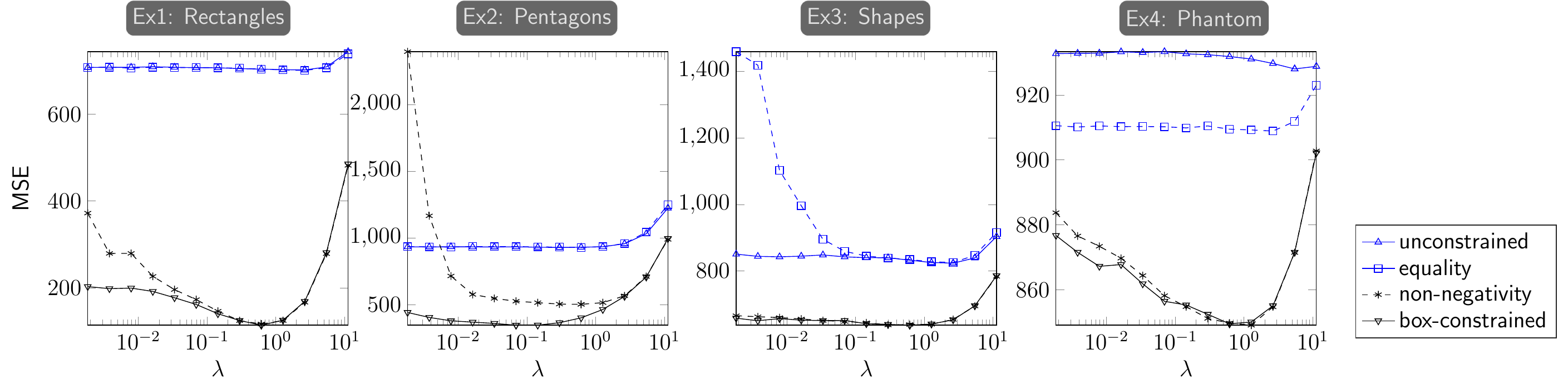}
	\end{center}
	\caption{Investigation on the impact of regularization parameter on MSE for different choices of constraints and different data sets. }\label{fig:plotObjectiveBLambda}
\end{figure}

\subsection{OED Results for Problem A} 
\label{sub:solution_approach_for_oed_problem_a}
Given the training data and a fine discretization of the tomography operator, we generate data, by evaluating the forward problem and adding 0.1\% Gaussian white noise. Then, the OED problem is solved twice. First, we aim at eliminating the number of rows in $\bfA$ by solving the OED problem A with an $\ell_1$-regularizer on the measurement parameters. Second, we adjust the weights of the non-zero weights by re-solving the OED problem without regularization. This procedure resembles the method introduced in~\cite{HaberEtAl2008OED}.

We first investigate the impact of the sparsity parameter $\beta$ on the design.  In the top row of Figure~\ref{fig:OEDA_beta}, we provide the number of projections $\ell$  for various values of $\beta$, and as expected, we see that with a larger sparsity parameter, we obtain fewer projections.  The more interesting result is in the second row where we provide the MSE as a function of $\beta.$  Here we see that even with fewer projections, reconstructions obtained by imposing constraints correspond to smaller MSE values.
\begin{figure}[t]
	\begin{center}
		\includegraphics[width=\textwidth]{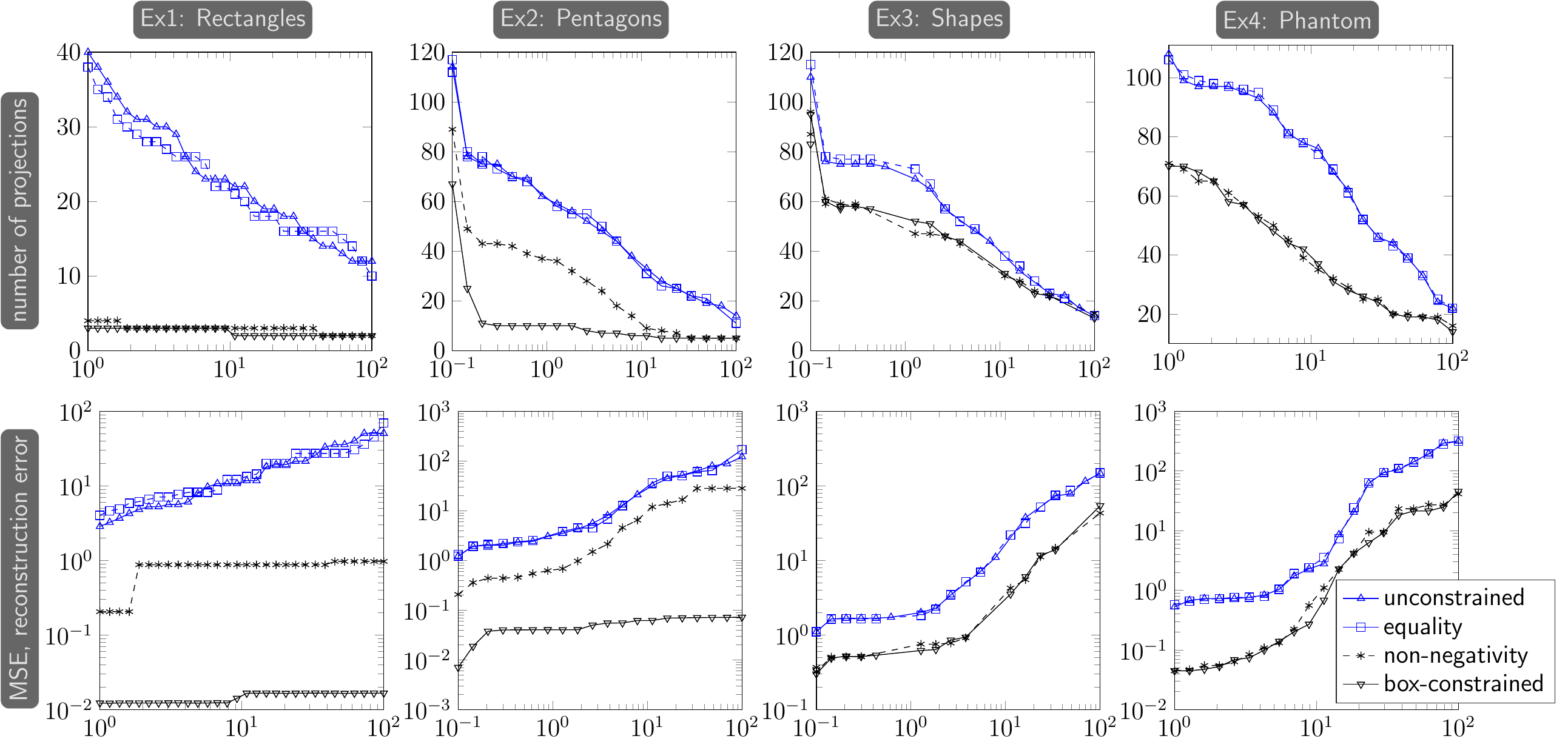}
	\end{center}
	\caption{Results for OED problem A using the four test data sets for the tomography problem (column wise). First row depicts the number of projection for the optimal design in dependence on the sparsity parameter $\beta$. The second row shows the optimal mean squared error (MSE).}
	\label{fig:OEDA_beta}
\end{figure}

In Figures~\ref{fig:OEDA_binary} and~\ref{fig:OEDA_blobs}, we provide four sample reconstructions and error images for each dataset and constraint. We observe that, overall, fewer projections are required and smaller reconstruction errors are possible if box or non-negativity constraints are included on the lower problem.
This distinction is most prominent with the datasets of binary images, where only a few projections are needed.

Intuitively, the optimal angles for the rectangle images should be 0 and 90 degrees and the optimal angles for the pentagon images should be 27, 63, 99, 135, and 171 degrees.  This is because the training images all share the same orientation, and angles orthogonal to the edges may be considered optimal (see Figure~\ref{fig:testData}). In Figure~\ref{fig:OEDA_binary}, we see that for the rectangles, the non-negative and box-constrained OED parameters are $\hat\bfp = [0,90]\t$ and $\hat\bfp = [0, 90]\t$  degrees respectively, and for the pentagons, $\hat\bfp = [25, 64, 99, 136, 171]\t$  and $\hat\bfp = [27, 62, 99, 134, 171]\t$ degrees respectively. These results illustrate the benefits of incorporating proper constraints to improve the optimal experimental design.

\begin{figure}
	\begin{center}
		\begin{center}
			\includegraphics[width=.9\textwidth]{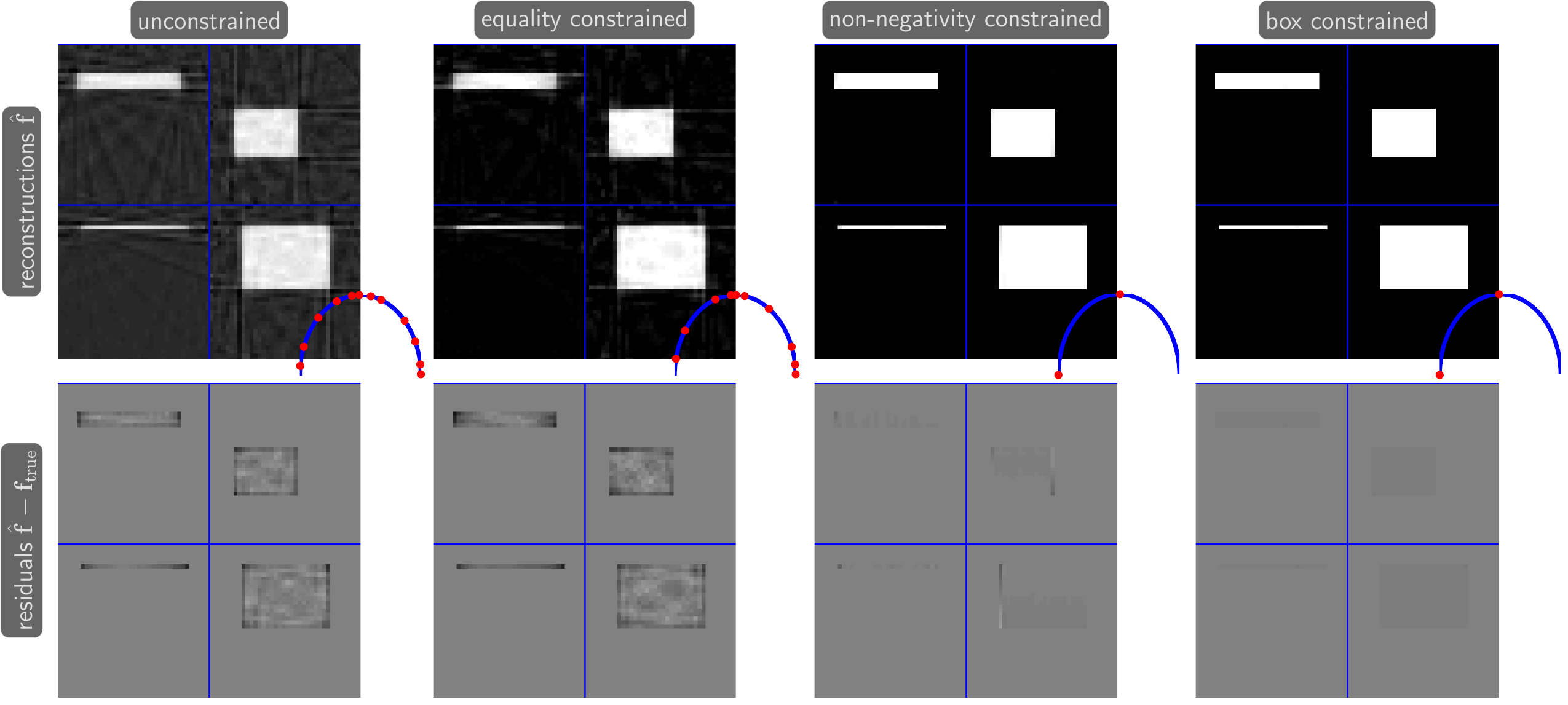}\\
			\includegraphics[width=.9\textwidth]{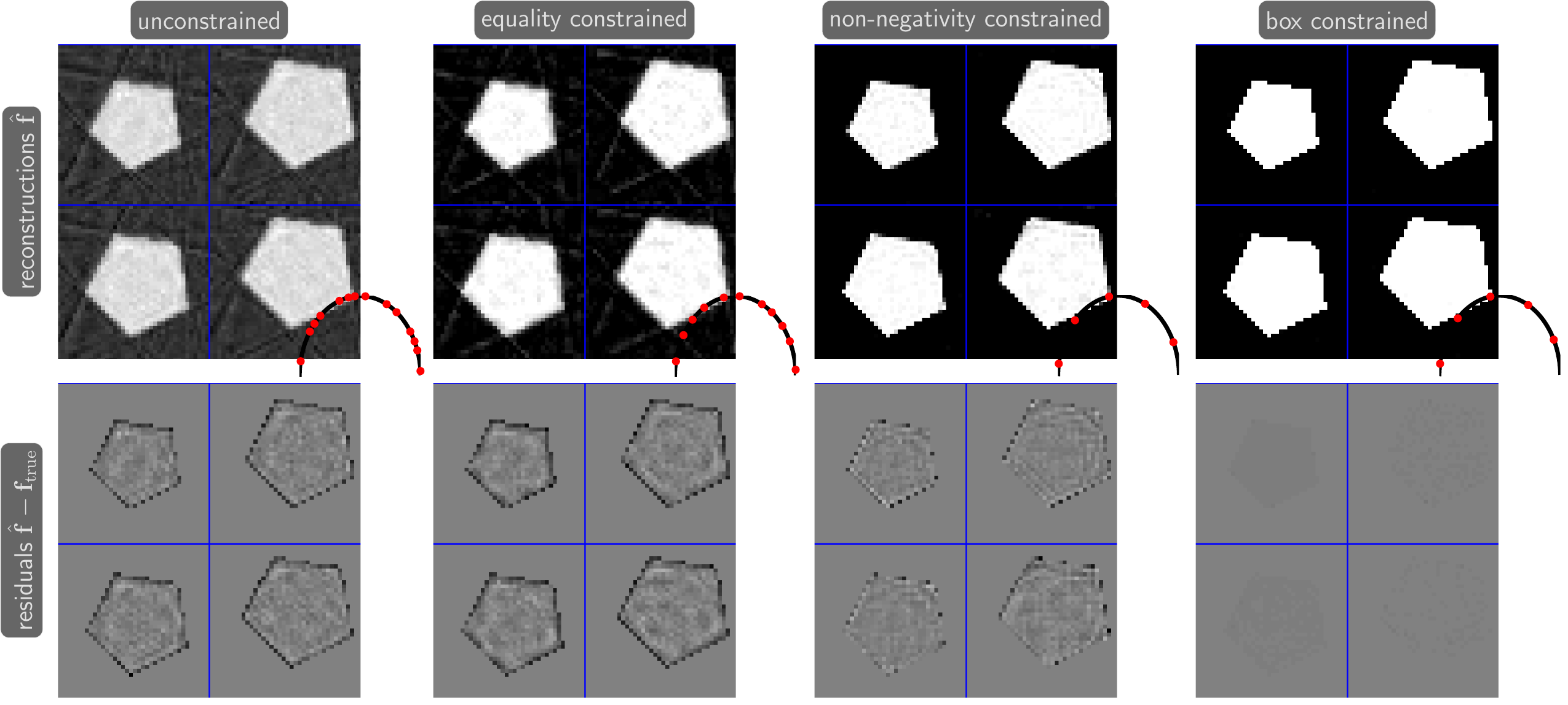}
		\end{center}
	\end{center}
	\caption{OED results for binary datasets (rectangles and pentagons respectively) comparing unconstrained, equality constrained, non-negativity constrained, and box constrained formulations of OED problem A.  Reconstructions and differences to the true training image for 4 randomly chosen images are provided, where the color axis is identical in each row to simplify comparison.  The red dots on the half circle correspond to the projection angles of the optimal designs. }
	\label{fig:OEDA_binary}
\end{figure}

\begin{figure}
	\begin{center}
		\begin{center}
			\includegraphics[width=.9\textwidth]{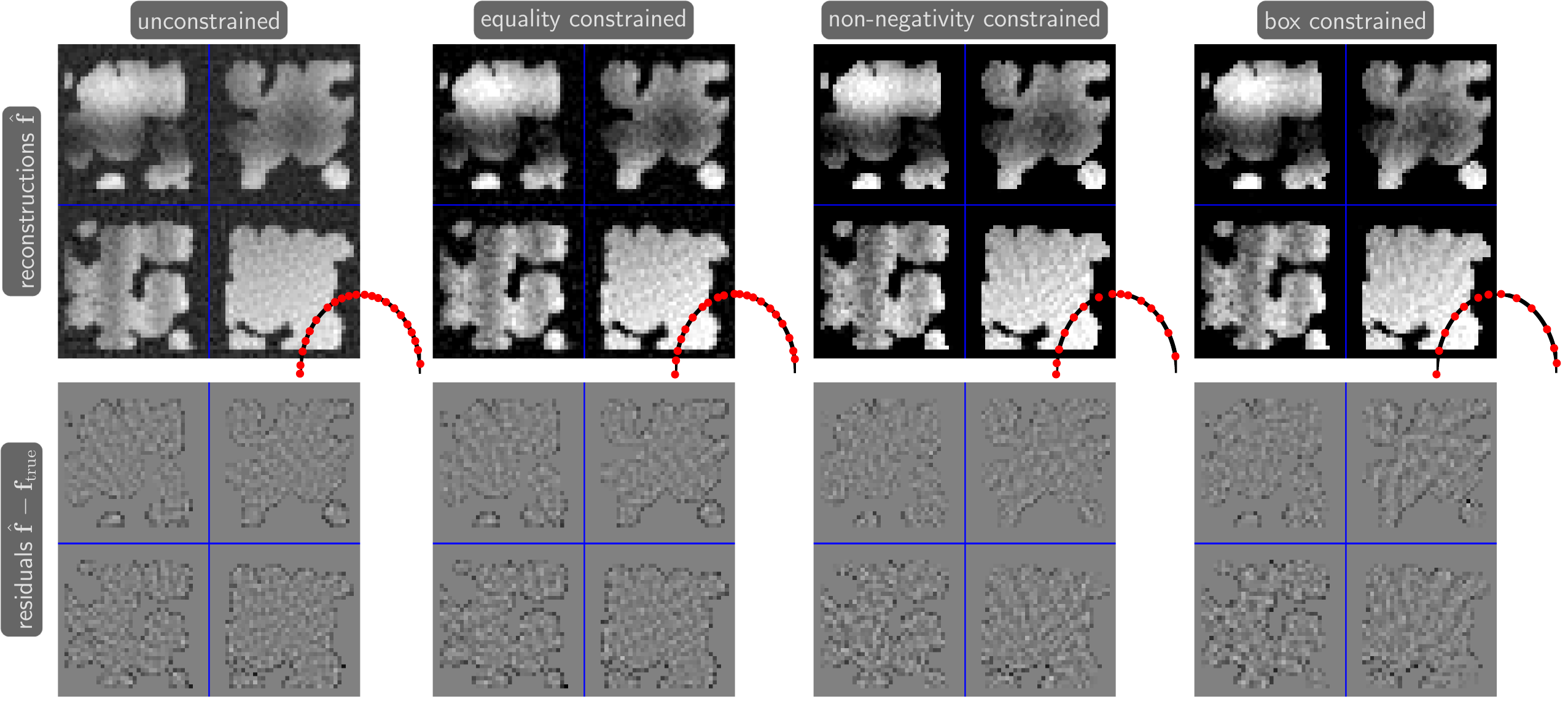}\\
			\includegraphics[width=.9\textwidth]{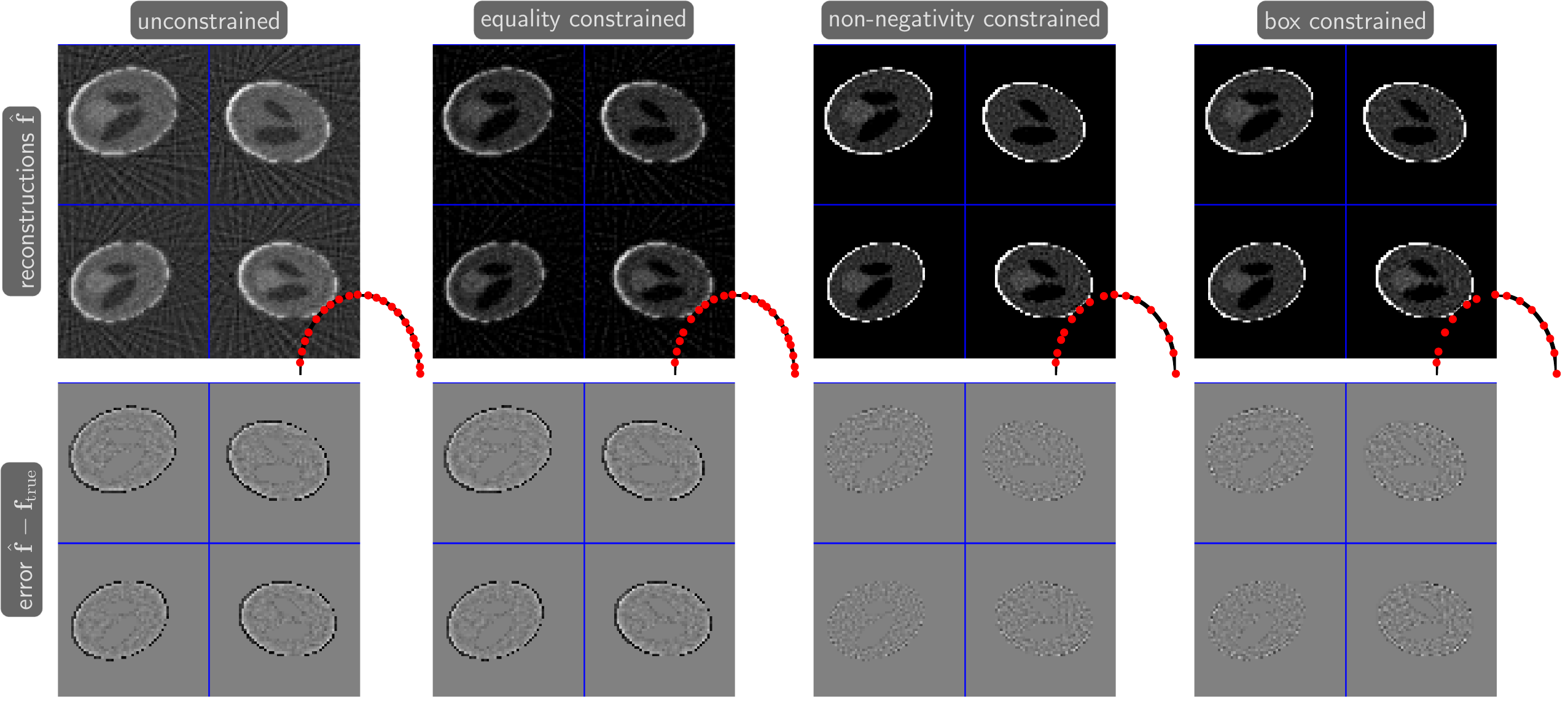}
		\end{center}
	\end{center}
	\caption{OED results for blobs and brain phantoms datasets comparing unconstrained, equality constrained, non-negativity constrained, and box constrained formulations of OED problem A.  Reconstructions and differences to the true training image for 4 randomly chosen images are provided, where the color axis is identical in each row to simplify comparison.  The red dots on the half circle correspond to the projection angles of the optimal designs. }
	\label{fig:OEDA_blobs}
\end{figure}

\section{Conclusions}
\label{sec:conclusions}
In this work, we consider optimal experimental design problems in the context of tomographic reconstruction and investigate the impact of state constraints on the optimal design.  We examine two problem formulations for enforcing sparse sampling in the design parameters, where OED problem A employs a sparsity enforcing regularization term and OED problem B achieves a desired level of sparsity by construction.  We investigate Bayes risk and empirical Bayes risk minimization techniques for OED.  For problems with known or well-approximated (e.g., from very large data sets) mean and covariance matrix, a reformulation of the Bayes risk can lead to efficient methods to obtain optimal designs for the unconstrained case.  However, the empirical risk minimization framework allows for incorporation of state constraints.  The primary challenge toward efficient optimization of the empirical problem is computing derivatives of the reconstructed image with respect to the design parameters.  We obtain these by using implicit differentiation of the KKT conditions within interior point methods and by exploiting parallel computing in Julia.  Our numerical results on various datasets demonstrate that including state constraints does indeed impact the optimal design, in that fewer projections are required, and smaller MSE values can be obtained. Some items for future work include considering integer or binary constraints on the design parameters, extensions to nonlinear inverse problems, and extensions to other applications.

\section*{Acknowledgments}
This work was initiated as a part of the SAMSI Program on Optimization 2016-2017.
The material was based upon work partially supported by the National Science Foundation under Grant DMS-1127914 to the Statistical and Applied Mathematical Sciences Institute. Any opinions, findings, and conclusions or recommendations expressed in this material are those of the author(s) and do not necessarily reflect the views of the National Science Foundation.
 \bibliographystyle{siamplain}

\end{document}